\RequirePackage{fix-cm}
\documentclass[smallextended,referee,envcountsect]{svjour3}       
\smartqed  

\usepackage{graphicx}
\usepackage{amsmath, ntheorem}
\usepackage{amssymb}
\usepackage{stackrel}
\usepackage{marvosym, enumerate}
\usepackage{amstext, amscd, latexsym}
\usepackage{amsfonts, ragged2e}
\usepackage{mathrsfs, pgfplots, enumerate, setspace}
\usepackage{subfigure}
\usepackage{cases}
\smartqed
\usepackage{cite}
\usepackage{amstext, graphicx, amscd, latexsym, amssymb}
\usepackage{amsfonts, ragged2e}
\usepackage{pgfplots, setspace}
\usepackage{geometry}
\usepackage{latexsym,bm}
\usepackage{float}
\usepackage{epstopdf,caption}
 
\usepackage{amsfonts,amsmath,amsbsy,amssymb,amscd,mathrsfs}
\usepackage{graphicx}
\usepackage{epstopdf}
\usepackage{algorithmic} 
\usepackage{graphicx,epstopdf}

 \newcommand{\nm}[1]{\left\lVert {#1} \right\rVert}
 
 \newcommand{\dual}[1]{\left\langle {#1} \right\rangle}

\usepackage[figuresright]{rotating}
\usepackage[sort&compress,numbers]{natbib}
\usepackage[ruled,linesnumbered]{algorithm2e} 
\usepackage[colorlinks,linkcolor=blue,anchorcolor=blue,citecolor=blue,urlcolor=blue]{hyperref}

\journalname{}

\begin{document}

\title{Scaled Proximal Gradient Methods for Multiobjective Optimization: Improved Linear Convergence and Nesterov's Acceleration}


\author{Jian Chen \and Liping Tang \and  Xinmin Yang  }

\institute{J. Chen \at National  Center  for  Applied  Mathematics in Chongqing, Chongqing Normal University, Chongqing 401331, China, and School of Mathematical Sciences, University of Electronic Science and Technology of China, Chengdu, Sichuan 611731, China\\
                    chenjian\_math@163.com\\
                   L.P. Tang \at National Center for Applied Mathematics in Chongqing, and School of Mathematical Sciences,  Chongqing Normal University, Chongqing 401331, China\\
                   tanglipings@163.com\\
        \Letter X.M. Yang \at National Center for Applied Mathematics in Chongqing,  Chongqing Normal University, Chongqing 401331, China\\
        xmyang@cqnu.edu.cn  \\}

\date{Received: date / Accepted: date}

\maketitle

\begin{abstract}
Over the past two decades, descent methods have received substantial attention within the multiobjective optimization field. Nonetheless, both theoretical analyses and empirical evidence reveal that existing first-order methods for multiobjective optimization converge slowly, even for well-conditioned problems, due to the objective imbalances. To address this limitation, we incorporate curvature information to scale each objective within the direction-finding subproblem, introducing a scaled proximal gradient method for multiobjective optimization (SPGMO). We demonstrate that the proposed method achieves improved linear convergence, exhibiting rapid convergence in well-conditioned scenarios. Furthermore, by applying small scaling to linear objectives, we prove that the SPGMO attains improved linear convergence for problems with multiple linear objectives. Additionally, integrating Nesterov's acceleration technique further enhances the linear convergence of SPGMO. To the best of our knowledge, this advancement in linear convergence is the first theoretical result that directly addresses objective imbalances in multiobjective first-order methods. Finally, we provide numerical experiments to validate the efficiency of the proposed methods and confirm the theoretical findings.

\keywords{Multiobjective optimization \and Proximal gradient method  \and Linear convergence \and Acceleration}
\subclass{90C29 \and 90C30}
\end{abstract}

\section{Introduction}
In multiobjective optimization, the primary goal is to optimize multiple objective functions simultaneously. Generally, it is not feasible to identify a single solution that simultaneously achieves optimal values for all objectives. Consequently, the notion of optimality is defined by {\it Pareto optimality} or {\it efficiency}. A solution is  Pareto optimal or efficient if no objective can be improved without sacrificing the others. As society and the economy progress, the applications of this type of problem have proliferated across a multitude of domains, such as engineering \cite{MA2004}, economics \cite{FW2014,TC2007}, management science \cite{E1984}, and machine learning \cite{SK2018}. 
\par  Solution strategies are pivotal in applications involving multiobjective optimization problems (MOPs). Over the past two decades, multiobjective gradient descent methods have gained escalating attention within the multiobjective optimization community, as they provide common descent directions for all objectives without requiring predefined parameters. Besides, this type of method enjoys an attractive property, as pointed out by Attouch et al. \cite{AGG2015}, in fields like game theory, economics, social science, and management: {\it it improves each of the objective functions}. As far as we know, the study of multiobjective gradient descent methods can be traced back to the pioneering works by Mukai \cite{M1980} and Fliege and Svaiter \cite{FS2000}. The latter elucidated that the multiobjective steepest descent direction reduces to the steepest descent direction when dealing with a single objective. This observation inspired researchers to extend ordinary numerical algorithms for solving MOPs (see, e.g., \cite{AP2021,BI2005,CL2016,FD2009,FV2016,GI2004,LP2018,MP2018,MP2019,P2014,QG2011} and references therein). 
\par Although multiobjective gradient descent methods are derived from their single-objective counterparts, a theoretical gap remains between the two types of approaches. Recently, Zeng et al. \cite{ZDH2019} and Fliege et al. \cite{FVV2019} established a linear convergence rate $\mathcal{O}(r^{k})$ of the steepest descent method for MOPs (SDMO), where $r=\sqrt{1-{\mu_{\min}}/{L_{\max}}}$, $\mu_{\min}:=\min\{\mu_{i}:i=1,2,\cdots,m\}$ and $L_{\max}:=\max\{L_{i}:i=1,2,\cdots,m\}$, with $\mu_{\min}$ and $L_{i}$ representing the strongly convex and smooth parameters, respectively. Tanabe et al. \cite{TFY2023} obtained a similar result for the proximal gradient method for MOPs (PGMO) \cite{TFY2019}. Notably, when minimizing a $\mu$-strongly convex and $L$-smooth function using the vanilla gradient method, the rate of convergence in terms of $\{\|x^{k}-x^{*}\|\}$ is $\sqrt{1-{\mu}/{L}}$. This highlights that {\it objective imbalances}, stemming from the substantially distinct curvature information of different objective functions, can result in a small value of ${\mu_{\min}}/{L_{\max}}$. Remarkably, even if each of the objective functions is not ill-conditioned (a relatively small ${L_{i}}/{\mu_{i}}$), the {\it imbalanced condition number} ${L_{\max}}/{\mu_{\min}}$ can be tremendous. This theoretical gap between first-order methods for single-objective optimization problems (SOPs) and MOPs elucidates why each objective is relatively easy to optimize individually but challenging when attempting to optimize them simultaneously. We emphasize that the objective imbalances are intrinsic to MOPs, particularly in large-scale and real-world scenarios, imposing significant challenges for solving MOPs via existing first-order methods \cite{FS2000,GI2004,LP2018,TFY2019,TFY2022}. This raises a critical question: How can the theoretical gap between first-order methods for SOPs and MOPs be bridged?
\par In this paper, we consider the question with the following generic model of unconstrained multiobjective composite optimization problem:
\begin{align*}
	\min\limits_{x\in\mathbb{R}^{n}} F(x), \tag{MCOP}\label{MCOP}
\end{align*}
where $F:\mathbb{R}^{n}\rightarrow(\mathbb{R}\cup+\{\infty\})^{m}$ is a vector-valued function. Each component $F_{i}$, $i=1,2,\cdots,m$, is defined by
$$F_{i}:=f_{i}+g_{i},$$ 
with convex and continuously differentiable function $f_{i}:\mathbb{R}^{n}\rightarrow\mathbb{R}$ and proper convex and lower semicontinuous but not necessarily differentiable function $g_{i}:\mathbb{R}^{n}\rightarrow\mathbb{R}\cup+\{\infty\}$. This type of problem finds wide applications in machine learning and statistics, and gradient descent methods tailored for it have received increasing attention (see, e.g., \cite{A2023,AFP2023,BMS2022,BG2018,TFY2022}). To address the aforementioned challenge, we incorporate curvature information to scale each objective within the direction-finding subproblem and propose a scaled proximal gradient method for multiobjective optimization (SPGMO). We analyze the convergence rates of SPGMO and offer new theoretical insights, helping to explain its observed rapid convergence in practical applications. The primary contributions of this paper are summarized as follows:
\par (i) When the smooth parameters are known, we apply them to scale each of objectives in the SPGMO and prove that its rate of convergence in terms of $\|x^{k}-x^{*}\|$ is $\sqrt{1-\min_{i=1,2,\cdots,m}\left\{{\mu_{i}}/{L_{i}}\right\}}$. The improved linear convergence bridges the theoretical gap between first-order methods for SOPs and MOPs, providing a theoretical basis for the superior performance of SPGMO over PGMO. Additionally, we establish improved linear convergence for SPGMO when applied to MOPs with both linear and strongly convex objectives. The convergence rate primarily depends on the strongly convex objectives, with the influence of linear objectives mitigated through small scaling. To our knowledge, this is the first result demonstrating linear convergence for descent methods on this class of problems.
\par (ii) When the smooth parameters are unknown, we investigate the selection of scaling parameters in SPGMO with line search. We demonstrate that these scaling parameters are crucial for the convergence rate of SPGMO with line search. Theoretical results indicate that optimal linear convergence is achieved by choosing either $\{\mu_{i}:i=1,2,...,m\}$ or $\{L_{i}:i=1,2,...,m\}$ as the scaling parameters. Notably, in practical applications, it is advisable to select each scaling parameter from the interval $[\mu_{i},L_{i}]$ to better capture the local geometry of the problem. Consequently, the Barzilai-Borwein method \cite{BB1988,CTY2023} emerges as a judicious choice for scaling parameters.
\par (iii) We propose an accelerated SPGMO (ASPGMO) which unifies the convex and strongly convex cases by employing different momentum parameters. In the context of convex problems, the ASPGMO recovers the well-known sub-linear convergence rate $\mathcal{O}(1/k^{2})$, a result initially rigorously demonstrated by Tanabe et al. \cite{TFY2022} for MOPs, with more recent advancements discussed in \cite{SP2022,SP2024}. For strongly convex problems, we demonstrate that ASPGMO, with an appropriately chosen momentum parameter, converges linearly at a rate of $1-\sqrt{\min_{i=1,2,\cdots,m}\left\{{\mu_{i}}/{L_{i}}\right\}}$. 
\par The remainder of the paper is organized as follows. In section \ref{sec2}, we present some necessary notations and definitions that will be used later. Section \ref{sec3} recalls the PGMO \cite{TFY2019}. In section \ref{sec4}, we propose a scaled proximal gradient method for MOPs with known smooth parameters and establish an improved linear convergence for strongly convex cases. Section \ref{sec5} is devoted to a scaled proximal gradient method for MOPs with unknown smooth parameters and discusses the selection of scaling parameters in the proposed approach. In section \ref{sec6}, we incorporate Nesterov's acceleration technique into SPGMO and analyze its convergence rates for both convex and strongly convex scenarios. The numerical experiments are presented in section \ref{sec7}, demonstrating the efficiency of the SPGMO and validating the theoretical results. Finally, we draw some conclusions at the end of the paper.

\section{Preliminaries}\label{sec2}
Throughout this paper, the $n$-dimensional Euclidean space $\mathbb{R}^{n}$ is equipped with the inner product $\langle\cdot,\cdot\rangle$ and the induced norm $\|\cdot\|$. We denote by $Jf(x)\in\mathbb{R}^{m\times n}$ the Jacobian matrix of $f$ at $x$, by $\nabla f_{i}(x)\in\mathbb{R}^{n}$ the gradient of $f_{i}$ at $x$.
We denote $[m]:=\{1,2,\cdots,m\}$, and $\Delta_{m}:=\{\lambda\in\mathbb{R}^{m}_{+}:\sum_{i\in[m]}\lambda_{i}=1\}$ the $m$-dimensional unit simplex. We define order relations $\preceq$ and $\prec$ in $\mathbb{R}^{m}$ as $$u\preceq v~\Leftrightarrow~v-u\in\mathbb{R}^{m}_{+}$$ and $$u\prec v~\Leftrightarrow~v-u\in\mathbb{R}^{m}_{++},$$ respectively.
We denote the level set of $F$ on $c\in\mathbb{R}^{m}$ as 
$$\mathcal{L}_{F}(c):=\{x:F(x)\preceq c\}.$$
\par Next, we introduce optimality concepts for (\ref{MCOP}) in the Pareto sense. 
\vspace{1mm}
\begin{definition}\label{def1}
	A vector $x^{\ast}\in\mathbb{R}^{n}$ is called Pareto solution to (\ref{MCOP}), if there exists no $x\in\mathbb{R}^{n}$ such that $F(x)\preceq F(x^{\ast})$ and $F(x)\neq F(x^{\ast})$.
\end{definition}
\vspace{1mm}
\begin{definition}\label{def2}
	A vector $x^{\ast}\in\mathbb{R}^{n}$ is called weakly Pareto solution to (\ref{MCOP}), if there exists no $x\in\mathbb{R}^{n}$ such that $F(x)\prec F(x^{\ast})$.
\end{definition}
\vspace{1mm}
\begin{definition}
	A differentiable function $h:\mathbb{R}^{n}\rightarrow\mathbb{R}$ is $L$-smooth if $$h(y)\leq h(x) + \dual{\nabla h(x),y-x}+\frac{L}{2}\|y-x\|^{2}$$ holds for
	all $x,y\in\mathbb{R}^{n}$. And $h$ is $\mu$-strongly convex if $$h(y)\geq h(x) + \dual{\nabla h(x),y-x}+\frac{\mu}{2}\|y-x\|^{2}$$ holds for all $x,y\in\mathbb{R}^{n}$.
\end{definition}

%

\section{Proximal gradient method for MCOPs}\label{sec3}
For the (\ref{MCOP}), Tanabe et al. \cite{TFY2019} proposed the following proximal gradient subproblem:
\begin{equation}\label{sub1}
	\min\limits_{x\in\mathbb{R}^{n}}\max\limits_{i\in[m]}\left\{{
		\left\langle\nabla f_{i}(x^{k}),d\right\rangle + g_{i}(x^{k}+d)-g_{i}(x^{k})}+\frac{\ell}{2}\|d\|^{2}\right\},
\end{equation}
where $\ell>0$. 
Assume that $f_{i}$ is $L_{i}$-smooth for $i\in[m]$, denote $L_{\max}:=\max\{L_{i}:i\in[m]\}$. When $\{L_{1},\cdots,L_{m}\}$ is available, the complete proximal gradient method for MCOPs is described as follows.
\begin{algorithm}\small  
	\caption{{\ttfamily{Proximal\_gradient\_method\_for\_MCOPs}}~\cite{TFY2019}}\label{pgmo} 
	\begin{algorithmic}[1]
		\REQUIRE{$x^{0}\in\mathbb{R}^{n},~\ell\geq{L_{\max}}$}
		\FOR{$k=0,...$}
		\STATE{Compute $d_{\ell}^{k}$ by solving subproblem (\ref{sub1})}
		\IF{$d_{\ell}^{k}=0$}
		\RETURN{$x^{k}$}
		\ELSE{
			\STATE{Update $x^{k+1}:= x^{k}+d_{\ell}^{k}$}  }
		\ENDIF
		\ENDFOR
	\end{algorithmic}
\end{algorithm}

When $m=1$, Algorithm \ref{pgmo} reduces to the proximal gradient method for SOPs, it is known that $\{x^{k}\}$ converges linearly to the unique minimizer when $f_{1}$ is strongly convex with $\mu_{1}>0$. Tanabe et al. \cite{TFY2023} also established the linear convergence for Algorithm \ref{pgmo}.
\begin{theorem}[See Theorem 5.3 of \cite{TFY2023}]\label{T3.1}
	Assume that $f_{i}$ is strongly convex with module $\mu_{i}>0$ and $L_{i}$-smooth for $i\in[m]$, let $\{x^{k}\}$ be the sequence generated by Algorithm \ref{pgmo}. Then there exists a Pareto solution $x^{*}$ such that 
	$$\nm{x^{k+1}-x^{*}}\leq\sqrt{1-\frac{\mu_{\min}}{L_{\max}}}\nm{x^{k}-x^{*}}.$$
\end{theorem}

Although Algorithm \ref{pgmo} enjoys linear convergence in strongly convex cases, the rate of convergence can be very slow even if each of the objectives is well-conditioned. We clarify the statement with the following example.
\vspace{1mm}
\begin{example}\label{exp1}
	Let us consider the following bi-objective problem:
	\begin{align*}
		\min\limits_{x\in \mathbb{R}^{2}}\left( F_{1}(x), F_{2}(x)\right),
	\end{align*}
	where $f_{1}(x)=\frac{1}{2}x_{1}^{2}+\frac{1}{2}x_{2}^{2}$, $f_{2}(x)=\frac{L}{2}x_{1}^{2}+\frac{L}{2}x_{2}^{2}$ ($L$ is a very large positive constant), $g_{1}(x)=g_{2}(x)=0$. It is easy to see that $L_{\max}=L$, condition numbers of $f_{1}$ and $f_{2}$ are both $1$, and the unique Pareto solution is $x^{*}=(0,0)^{T}$. For any noncritical point $x^{k}$, we derive that $x^{k+1}=(1-{1}/{L})x^{k}$, i.e., $$\nm{x^{k+1}-x^{*}}=\left(1-\frac{1}{L}\right)\nm{x^{k}-x^{*}}.$$
	Recall that $L$ is a very large positive constant, thus the rate of linear convergence is very slow.
\end{example}
\section{Scaled proximal gradient method for MCOPs}\label{sec4}

In this section, we propose a scaled proximal gradient method for (\ref{MCOP}) with improved linear convergence. Firstly, we emphasize that the slow convergence of Algorithm \ref{pgmo} is mainly due to the global upper bound used in (\ref{sub1}):
$$F_{i}(x^{k}+d)-F_{i}(x^{k})\leq\left\langle\nabla f_{i}(x^{k}),d\right\rangle+g_{i}(x^{k}+d)-g_{i}(x^{k})+\frac{L_{\max}}{2}\|d\|^{2}$$ for all $i\in[m]$, which may be too conservative for objectives with small global smoothness parameters. Instead, we employ a separate global smoothness parameter for each objective to devise the following scaled proximal gradient subproblem:
\begin{equation}\label{sub2}
	\min\limits_{x\in\mathbb{R}^{n}}\max\limits_{i\in[m]}\left\{\frac{
		\left\langle\nabla f_{i}(x^{k}),d\right\rangle + g_{i}(x^{k}+d)-g_{i}(x^{k})}{L_{i}}+\frac{1}{2}\|d\|^{2}\right\}.
\end{equation}
This strategy can capture the curvature information of each objective and alleviate interference among the objectives. The complete scaled proximal gradient method for MCOPs (SPGMO) is described as follows. 

\begin{algorithm}\small  
	\caption{{\ttfamily{Scaled\_proximal\_gradient\_method\_for\_MCOPs}}~\cite{TFY2019}}\label{spgmo} 
	\begin{algorithmic}[1]
		\REQUIRE{$x^{0}\in\mathbb{R}^{n}$}
		\FOR{$k=0,...$}
		\STATE{Compute $d^{k}$ by solving subproblem (\ref{sub2})}
		\IF{$d^{k}=0$}
		\RETURN{$x^{k}$}
		\ELSE{
			\STATE{Update $x^{k+1}:= x^{k}+d^{k}$}  }
		\ENDIF
		\ENDFOR
	\end{algorithmic}
\end{algorithm}

Let us denote the following scaled multiobjective composite optimization problem:
\begin{align*}
	\min\limits_{x\in\mathbb{R}^{n}} F^{L}(x), \tag{MCOP$_{L}$}\label{MCOP2}
\end{align*}
where ${F}^{L}:\mathbb{R}^{n}\rightarrow(\mathbb{R}\cup\{+\infty\})^{m}$ is a vector-valued function. Each component $F^{L}_{i}$, $i\in[m]$, is defined by
$$F^{L}_{i}:=\frac{f_{i}+g_{i}}{L_{i}}.$$ 
By scaling the objective functions, we note that Algorithm \ref{spgmo} for (\ref{MCOP}) can be interpreted as Algorithm \ref{pgmo} for (\ref{MCOP2}). As a result, the convergence rates analysis for Algorithm \ref{pgmo} in \cite{TFY2023} also applies to Algorithm \ref{spgmo}. Before we proceed with the convergence analysis, we will first present several properties of  (\ref{MCOP2}).
\begin{proposition}\label{fl}
	For the {\rm(\ref{MCOP2})}, the following statements hold.
	\begin{itemize}
		\item[$\mathrm{(i)}$]  $x\in\mathbb{R}^{n}$ is a Pareto critical point of {\rm(\ref{MCOP2})} if and only if $x\in\mathbb{R}^{n}$ is a Pareto critical point of {\rm(\ref{MCOP})}.
		\item[$\mathrm{(ii)}$]  For $i\in[m]$, if $f_{i}$ is $L_{i}$-smooth, then ${f_{i}}/{L_{i}}$ is $1$-smooth.
		\item[$\mathrm{(iii)}$]  For $i\in[m]$, if $f_{i}$ is strongly convex with $\mu_{i}\geq0$, then ${f_{i}}/{L_{i}}$ is strongly convex with ${\mu_{i}}/{L_{i}}$.
	\end{itemize}
\end{proposition}
\begin{proof}
	The proof can be easily derived by the definition of $F^{L}$, we omit it here.
\end{proof}
\subsection{Improved linear convergence of SPGMO}
Firstly, we establish a fundamental inequality.
\begin{lemma}\label{l5.9}
	Assume that $f_{i}$ is strongly convex with modulus $\mu_{i}\geq0$ and $L_{i}$-smooth for $i\in[m]$. Let $\{x^{k}\}$ be the sequence generated by Algorithm \ref{spgmo}. Then, there exists $\lambda^{k}\in\Delta_{m}$ such that
	\begin{equation}\label{fineq}
		\begin{aligned}
			&~~~~\sum\limits_{i\in[m]}\lambda^{k}_{i}\frac{F_{i}(x^{k+1})-F_{i}(x)}{L_{i}}\\
			&\leq\frac{1}{2}\|x^{k}-x\|^{2}-\frac{1}{2}\|x^{k+1}-x\|^{2}-\sum\limits_{i\in[m]}\lambda^{k}_{i}\frac{\mu_{i}}{2L_{i}}\|x^{k}-x\|^{2},~{\rm for~all}~ x\in\mathbb{R}^{n}.
		\end{aligned}
	\end{equation}
\end{lemma}
\begin{proof}
	The assertions can be obtained by using similar arguments as in the proof of \cite[Lemma 5.2]{TFY2020}
\end{proof}
\begin{theorem}\label{T5}
	Assume that $f_{i}$ is strongly convex with modulus $\mu_{i}>0$ and $L_{i}$-smooth for $i\in[m]$. Let $\{x^{k}\}$ be the sequence generated by Algorithm \ref{spgmo} and $x^{*}$ be the Pareto solution satisfies $F(x^{*})\preceq F(x^{k})$ for all $k\geq0$. Then we have 
	$$\|x^{k+1}-x^{*}\|\leq\sqrt{1-\min\limits_{i\in[m]}\left\{\frac{\mu_{i}}{ L_{i}}\right\}}\|x^{k}-x^{*}\|.$$
\end{theorem}
\vspace{1mm}
\begin{proof}
	By substituting $x=x^{*}$ into inequality (\ref{fineq}), we obtain
	\begin{align*}
		&~~~~\sum\limits_{i\in[m]}\lambda^{k}_{i}\frac{F_{i}(x^{k+1})-F_{i}(x^{*})}{L_{i}}\\
		&\leq\frac{1}{2}\|x^{k}-x^{*}\|^{2}-\frac{1}{2}\|x^{k+1}-x^{*}\|^{2}-\sum\limits_{i\in[m]}\lambda^{k}_{i}\frac{\mu_{i}}{2L_{i}}\|x^{k}-x^{*}\|^{2}~{\rm for~all}~ x\in\mathbb{R}^{n}.
	\end{align*}
	Applying $F(x^{*})\preceq F(x^{k})$, it follows that
	\begin{equation}\label{nsc}
		\|x^{k+1}-x^{*}\|\leq\sqrt{1-\sum\limits_{i\in[m]}\lambda^{k}_{i}\frac{\mu_{i}}{L_{i}}}\|x^{k}-x^{*}\|\leq\sqrt{1-\min\limits_{i\in[m]}\left\{\frac{\mu_{i}}{L_{i}}\right\}}\|x^{k}-x^{*}\|,
	\end{equation}
	where the last inequality holds due to the fact $\lambda^{k}\in\Delta_{m}$.
\end{proof}
\vspace{1mm}
\begin{remark}
	In the following, we present two remarks on the linear convergence results of the SPGMO and PGMO.
	\begin{itemize}
		\item[(i)] When $m=1$, both the SPGMO and the PGMO reduce to the proximal gradient method for SOPs. In cases where $m\neq1$, the condition numbers $\tilde{\kappa}:=L_{\max}/\mu_{\min}$ and $\kappa:=\max_{i\in[m]}\{L_{i}/\mu_{i}\}$ are pivotal for the geometric convergence of PGMO and SPGMO, respectively. Notably, since $\tilde{\kappa}\geq\kappa$, the SPGMO exhibits improved linear convergence compared to the PGMO. Furthermore, the SPGMO exhibits rapid linear convergence provided that all differentiable components are not ill-conditioned. Conversely, in such scenarios, the PGMO may experience slow convergence due to objective imbalances (see Example \ref{exp1}).
		\item[(ii)] Objective imbalances pose a significant challenge to existing first-order methods in multiobjective optimization, particularly in large-scale and real-world scenarios where objective imbalances are intrinsic to MOPs. To quantify the objective imbalances in strongly convex cases, we introduce the following parameter:
		\begin{equation}\label{oimb}
			\zeta:=\frac{\tilde{\kappa}}{\kappa}.
		\end{equation}
		A larger value of this parameter indicates a greater degree of objective imbalances in the problem.
		\item[(iii)] As outlined in Figure \ref{fl}, for the function $F^{L}$ we have $L_{\max}=1$ and $\mu_{\min}=\min_{i\in[m]}\{\mu_{i}/L_{i}\}$. Consequently, Theorem \ref{T5} follows as a corollary of Theorem \ref{T3.1}, resulting in an objective imbalances parameter $\zeta$ of $F^{L}$ is $1$.
	\end{itemize}
\end{remark}

\subsection{Linear convergence with some linear objective functions}
In addition to strongly convex cases, linear objectives often introduce significant objective imbalances in MOPs, decelerating the convergence of first-order methods \cite[Examples 2 and 3]{CTY2023}. In what follows, we confirm that a small scaling parameter is beneficial to mitigate the influence of the linear objectives from a theoretical perspective. Consider the linear constrained MOP with linear objective functions, which is described as follows:

\begin{align*}
	\min\limits_{x\in \mathcal{X}} f(x),\tag{LCMOP}\label{LCMOP}
\end{align*}
where $f:\mathbb{R}^{n}\mapsto\mathbb{R}^{m}$ is a vector-valued function; the component $f_{i}$ is linear for $i\in\mathcal{L}$, and $\mu_{i}$-strongly convex and $L_{i}$-smooth for $i\in[m]\setminus\mathcal{L}$, respectively; $\mathcal{X}=\{x:Ax\leq a,~Bx=b\}$ with $A\in\mathbb{R}^{\mathbb{|\mathcal{J}|}\times n},~B\in\mathbb{R}^{\mathbb{|\mathcal{E}|}\times n}$. This type of problem has wide applications in portfolio selection \cite{M1952} and can be reformulated as (\ref{MCOP}) with $F_{i}=f_{i}+\mathbb{I}_{\mathcal{X}}$ for $i\in[m]$, where 
\begin{equation*}
	\mathbb{I}_{\mathcal{X}}(x):=\left\{
	\begin{aligned}
		0,~~~~~~&x\in\mathcal{X},\\
		+\infty,~~~&x\notin\mathcal{X}.
	\end{aligned}
	\right.
\end{equation*}
\par Regarding the first inequality in (\ref{nsc}), the linear convergence rate in terms of $\|x^{k}-x^{*}\|$ is actually $\sqrt{1-\sum_{i\in[m]}\lambda^{k}_{i}\frac{\mu_{i}}{L_{i}}}$ for  Algorithm \ref{spgmo}. Consequently, Algorithm \ref{spgmo} exhibits linear convergence as long as $\sum_{i\in[m]}\lambda^{k}_{i}\frac{\mu_{i}}{L_{i}}>0$. A similar statement holds for PGMO. However, such a statement holds under the restrictive condition $\sum_{i\in[m]}\lambda^{k}_{i}\frac{\mu_{i}}{L_{i}}>0$, which is equivalent to $\sum_{i\in[m]\setminus\mathcal{L}}\lambda^{k}_{i}>0$. In what follows, we will prove that a small value of $L_{i}$ for $i\in\mathcal{L}$ leads to larger value of $\sum_{i\in[m]\setminus\mathcal{L}}\lambda^{k}_{i}$ in SPGMO.

When minimizing (\ref{LCMOP}) using SPGMO, the subproblem (\ref{sub2}) can be expressed as follows:

\begin{align*}
	\min&~~~~t +\frac{1}{2}\nm{d}^{2}         \\
	{\rm s.t.}&\dual{\frac{\nabla f_{i}(x^{k})}{L_{i}},d}\leq t,~i\in[m],\\
	&A(x^{k}+d)\leq a,\\
	&Bd=0,
\end{align*}
where $L_{i}=\alpha_{\min}$ for $i\in\mathcal{L}$, and $\alpha_{\min}$ is a very small positive constant\footnote{The upper bound $f_{i}(y)\leq f_{i}(x)+\dual{\nabla f_{i}(x),y-x}\leq L_{i}\nm{y-x}^{2}$ holds for $i\in\mathcal{L}$ with any $L_{i}
	\geq0$,.}.
By KKT conditions, we obtain $$d^{k}=-\left(\sum\limits_{i\in[m]}\lambda^{k}_{i}\frac{\nabla f_{i}(x^{k})}{L_{i}}+\sum\limits_{j\in\mathcal{J}}\eta^{k}_{j}A_{j}+\sum\limits_{e\in\mathcal{E}}\xi^{k}_{e}B_{e}\right),$$ where $A_{j}^{T}$ and $B_{e}^{T}$ is the $j$-th and $e$-th row of $A$ and $B$, respectively. The vector $(\lambda^{k},\eta^{k},\xi^{k})\in\mathbb{R}^{m+|\mathcal{J}|+|\mathcal{E}|}$ is a solution of the following Lagrangian dual problem:

\begin{align*}
	\min&~~~~\frac{1}{2}\nm{\sum\limits_{i\in[m]}\lambda_{i}\frac{\nabla f_{i}(x^{k})}{L_{i}}+\sum\limits_{j\in\mathcal{J}}\eta_{j}A_{j}+\sum\limits_{e\in\mathcal{E}}\xi_{e}B_{e}}^{2} - \sum\limits_{j\in\mathcal{J}}\eta_{j}\left(\dual{A_{j},x^{k}}-a_{j}\right)     \\
	{\rm s.t.}&~~~~(\lambda,\eta,\xi)\in\Delta_{m}\times\mathbb{R}^{|\mathcal{J}|}_{+}\times\mathbb{R}^{|\mathcal{E}|}.
\end{align*}
And complementary slackness condition gives that
\begin{equation}\label{cse}
	\eta^{k}_{j}\left(\dual{A_{j},x^{k}+d^{k}}-a_{j}\right)=0~{\rm for~all}~ j\in\mathcal{J}.
\end{equation}
Denote $\mathcal{J}^{k}:=\{j\in\mathcal{J}:\dual{A_{j},x^{k}}=a_{j}\}$, this together with (\ref{cse}) and the fact that $x^{k+1}=x^{k}+d^{k}$ implies 
$$d^{k}=-\left(\sum\limits_{i\in[m]}\lambda^{k}_{i}\frac{\nabla f_{i}(x^{k})}{L_{i}}+\sum\limits_{j\in\mathcal{J}^{k+1}}\eta^{k}_{j}A_{j}+\sum\limits_{e\in\mathcal{E}}\xi^{k}_{e}B_{e}\right).$$
Before presenting the linear convergence of SPGMO for (\ref{LCMOP}),  let us first define the following multiobjective linear programming problem:
\begin{align*}
	\min\limits_{x\in \mathcal{X}} f_{\mathcal{L}}(x),\tag{MLP}\label{MLP}
\end{align*}
where $f_{\mathcal{L}}:\mathbb{R}^{n}\rightarrow\mathbb{R}^{|\mathcal{L}|}$ is the linear component of $f$.
As a result, every weakly Pareto solution of (\ref{MLP}) is also a weakly Pareto solution of (\ref{LCMOP}). 
\vspace{1mm}
\begin{proposition}\label{pl1}
	Let $\{x^{k}\}$ be the sequence generated by Algorithm \ref{spgmo} with $L_{i}=\alpha_{\min},~i\in\mathcal{L}$ for (\ref{LCMOP}) and $x^{*}$ be the weakly Pareto solution satisfies $F(x^{*})\preceq F(x^{k})$ for all $k\geq0$. If $x^{*}$ is not a weakly Pareto solution of (\ref{MLP}), then $$\|x^{k+1}-x^{*}\|\leq\sqrt{1-(1-c_{1}\alpha_{\min})\min\limits_{i\in[m]\setminus\mathcal{L}}\left\{\frac{\mu_{i}}{ L_{i}}\right\}}\|x^{k}-x^{*}\|,$$
	where $c_{1}:=\max\limits_{i\in[m]\setminus\mathcal{L}}\frac{2(\|\nabla f_{i}(x^{0})\| + L_{i}R)}{\epsilon L_{i}}$.
\end{proposition} 
\vspace{1mm}
\begin{proof}
	We refer to Theorem \ref{T5}(ii), which states
	\begin{equation}\label{lce}
		\|x^{k+1}-x^{*}\|\leq\sqrt{1-\sum\limits_{i\in[m]}\lambda^{k}_{i}\frac{\mu_{i}}{L_{i}}}\|x^{k}-x^{*}\|.
	\end{equation} 
	Given that $x^{*}$ is not a weakly Pareto solution of (\ref{MLP}), we deduce $0\notin C^{k}:=\{\sum_{i\in\mathcal{L}}\lambda_{i} g_{i}+\sum_{j\in\mathcal{J}^{k}}\eta_{j}A_{j}+\sum_{e\in\mathcal{E}}\xi_{e}B_{e}:\lambda\in\Delta_{|\mathcal{L}|},\eta\in\mathbb{R}^{|\mathcal{J}^{k}|}_{+},\xi\in\mathbb{R}^{|\mathcal{E}|}\}$, where $g_{i}$ represents the gradient of the linear function $f_{i},~i\in\mathcal{L}$. Furthermore, since  $\mathcal{J}^{k}\subset\mathcal{J}$, we can infer that $|\{C^{k}\}|$ is finite. Consequently, we can define $\epsilon:=\min_{x\in \{C^{k}\}}\|x\|>0$. By direct calculation, we have
	\begin{equation}\label{e5.13}\footnotesize
		\begin{aligned}
			&\|d^{k}\|=\nm{\sum\limits_{i\in[m]}\lambda^{k}_{i}\frac{\nabla f_{i}(x^{k})}{L_{i}}+\sum\limits_{j\in\mathcal{J}^{k+1}}\eta^{k}_{j}A_{j}+\sum\limits_{e\in\mathcal{E}}\xi^{k}_{e}B_{e}}\\
			&\geq\nm{\sum\limits_{i\in\mathcal{L}}\lambda^{k}_{i}\frac{g_{i}}{L_{i}}+\sum\limits_{j\in\mathcal{J}^{k+1}}\eta^{k}_{j}A_{j}+\sum\limits_{e\in\mathcal{E}}\xi^{k}_{e}B_{e}}-\nm{\sum\limits_{i\in[m]\setminus\mathcal{L}}\lambda^{k}_{i}\frac{\nabla f_{i}(x^{k})}{L_{i}}}\\
			&=\frac{\sum\limits_{i\in\mathcal{L}}\lambda^{k}_{i}}{\alpha_{\min}}\nm{\sum\limits_{i\in\mathcal{L}}\frac{\lambda^{k}_{i}}{\sum\limits_{i\in\mathcal{L}}\lambda^{k}_{i}}g_{i}+\sum\limits_{j\in\mathcal{J}^{k+1}}\frac{\alpha_{\min}\eta^{k}_{j}}{\sum\limits_{i\in\mathcal{L}}\lambda^{k}_{i}}A_{j}+\sum\limits_{e\in\mathcal{E}}\frac{\alpha_{\min}\xi^{k}_{e}}{\sum\limits_{i\in\mathcal{L}}\lambda^{k}_{i}}B_{e}}-\nm{\sum\limits_{i\in[m]\setminus\mathcal{L}}\lambda^{k}_{i}\frac{\nabla f_{i}(x^{k})}{L_{i}}}\\
			&\geq\frac{\epsilon\sum\limits_{i\in\mathcal{L}}\lambda^{k}_{i}}{\alpha_{\min}} -\nm{\sum\limits_{i\in[m]\setminus\mathcal{L}}\lambda^{k}_{i}\frac{\nabla f_{i}(x^{k})}{L_{i}}},
		\end{aligned}
	\end{equation}
	where the second equality follows from the fact that $L_{i}=\alpha_{\min}$ for all $i\in\mathcal{L}$, and the last inequality is given by the definition of $\epsilon$.
	By simple calculation, we have 
	\begin{align*}
		\frac{1}{2}\|d^{k}\|^{2}&=\frac{1}{2}\nm{\sum\limits_{i\in[m]}\lambda^{k}_{i}\frac{\nabla f_{i}(x^{k})}{L_{i}}+\sum\limits_{j\in\mathcal{J}}\eta^{k}_{j}A_{j}+\sum\limits_{e\in\mathcal{E}}\xi^{k}_{e}B_{e}}^{2}\\
		&\leq\frac{1}{2}\nm{\sum\limits_{i\in[m]}\lambda^{k}_{i}\frac{\nabla f_{i}(x^{k})}{L_{i}}+\sum\limits_{j\in\mathcal{J}}\eta^{k}_{j}A_{j}+\sum\limits_{e\in\mathcal{E}}\xi^{k}_{e}B_{e}}^{2}- \sum\limits_{j\in\mathcal{J}}\eta^{k}_{j}\left(\dual{A_{j},x^{k}}-a_{j}\right) \\
		&\leq\frac{1}{2}\max\limits_{i\in[m]\setminus\mathcal{L}}\nm{\frac{\nabla f_{i}(x^{k})}{L_{i}}}^{2},
	\end{align*}
	where the second inequality follows by the non-negativeness of $\eta^{k}_{j}$ and $a_{j}-\dual{A_{j},x^{k}}$, and the last inequality is due to the fact that $(\lambda^{k},\eta^{k},\xi^{k})$ is a solution of the dual problem.
	This together with (\ref{e5.13}) implies
	$$\frac{\epsilon\sum\limits_{i\in\mathcal{L}}\lambda^{k}_{i}}{\alpha_{\min}}\leq\max\limits_{i\in[m]\setminus\mathcal{L}}\nm{\frac{\nabla f_{i}(x^{k})}{L_{i}}}+\nm{\sum\limits_{i\in[m]\setminus\mathcal{L}}\lambda^{k}_{i}\frac{\nabla f_{i}(x^{k})}{L_{i}}}\leq2\max\limits_{i\in[m]\setminus\mathcal{L}}\nm{\frac{\nabla f_{i}(x^{k})}{L_{i}}}.$$
	Notice that $f_{i}$ is strongly convex for $i\in[m]\setminus\mathcal{L}$, then we deduce that $\mathcal{L}_{F}(F(x^{0}))$ is bounded. Denoting $R:=\max\{\nm{x-y}:x,y\in\mathcal{L}_{F}(F(x^{0}))\}$, and utilizing the $L_{i}$-smoothness of $f_{i},i\in[m]\setminus\mathcal{L}$, we derive an upper bound of $\|\nabla f_{i}(x^{k})\|$:
	\begin{align*}
		\|\nabla f_{i}(x^{k})\|&\leq\|\nabla f_{i}(x^{0})\|+\|\nabla f_{i}(x^{k})-\nabla f_{i}(x^{0})\|\leq \|\nabla f_{i}(x^{0})\| + L_{i}R.
	\end{align*}
	Therefore, we obtain
	\begin{equation}\label{dse}
		\sum\limits_{i\in\mathcal{L}}\lambda^{k}_{i}\leq\max\limits_{i\in[m]\setminus\mathcal{L}}\frac{2(\|\nabla f_{i}(x^{0})\| + L_{i}R)}{\epsilon L_{i}}\alpha_{\min}.
	\end{equation}
	Substituting the preceding bound into (\ref{lce}), the desired result follows.
\end{proof}
\vspace{1mm}

\vspace{1mm}
\begin{remark}
	From Proposition \ref{pl1}, the sum of dual variables of linear objectives $\sum_{i\in\mathcal{L}}\lambda^{k}_{i}$ is dominated by $\alpha_{\min}$. Consequently, the impact of the linear objectives can be effectively mitigated by employing a small scaling parameter $\alpha_{\min}$ in the direction-finding subproblems. 
\end{remark}
\vspace{1mm}
\begin{remark}
	The assumption $x^{*}$ is not a weakly Pareto solution of (\ref{MLP}) seems restrictive. In practice, the Pareto set of (\ref{LCMOP}) can be a $(m-1)$-dimensional manifold, and the Pareto set of (\ref{MLP}) can be a $(|\mathcal{L}|-1)$-dimensional sub-manifold within the $(m-1)$-dimensional manifold. As a result, for a random initial point $x^{0}$, the probability that $x^{*}$ is not a weakly Pareto solution of (\ref{MLP}) can be $1$. Additionally, when $|\mathcal{J}|=0$, meaning (\ref{MLP}) has no inequality constraints, then the problem either has no weakly Pareto solution or every feasible point is a weakly Pareto solution. In other words, SPGMO converges linearly for equality constrained and unconstrained (\ref{LCMOP}). 
\end{remark}
\vspace{1mm}
\par In the following, we also analyze the convergence rate of PGMO for (\ref{LCMOP}).
\vspace{1mm}
\begin{proposition}\label{pl2}
	Let $\{x^{k}\}$ be the sequence generated by Algorithm \ref{pgmo} for (\ref{LCMOP}) and $x^{*}$ be the weakly Pareto solution satisfies $F(x^{*})\preceq F(x^{k})$ for all $k\geq0$. If $x^{*}$ is not a weakly Pareto solution of (\ref{MLP}), then there exists $\kappa\in(0,1)$ such that $\|x^{k}-x^{*}\|\leq\kappa^{k}\|x^{0}-x^{*}\|.$
\end{proposition}
\vspace{1mm}
\begin{proof}
	Setting $L_{i}=L_{\max}$ in SPGMO, it coincides with PGMO. Consequently, (\ref{lce}) collapses to
	\begin{small}
		\begin{equation}\label{e5.15}
			\|x^{k+1}-x^{*}\|\leq\sqrt{1-\sum\limits_{i\in[m]}\lambda^{k}_{i}\frac{\mu_{i}}{L_{\max}}}\|x^{k}-x^{*}\|.
		\end{equation}
	\end{small}
	Substituting $L_{i}=L_{\max}$ into (\ref{e5.13}), it follows that
	\begin{small}
		\begin{equation*}
			\begin{aligned}
				\epsilon\sum\limits_{i\in\mathcal{L}}\lambda^{k}_{i}&\leq\nm{\sum\limits_{i\in[m]\setminus\mathcal{L}}\lambda^{k}_{i}\nabla f_{i}(x^{k})}+L_{\max}\|d^{k}\|\\
				&\leq\left(\sum\limits_{i\in[m]\setminus\mathcal{L}}\lambda^{k}_{i}\right)\max\limits_{i\in[m]\setminus\mathcal{L}}\left\{\|\nabla f_{i}(x^{0})\| + L_{i}R\right\}+L_{\max}\|d^{k}\|.
			\end{aligned}
		\end{equation*}
	\end{small}
	Rearranging the above inequality, we have
	$$\sum\limits_{i\in\mathcal{L}}\lambda^{k}_{i}\leq\frac{\max\limits_{i\in[m]\setminus\mathcal{L}}\left\{\|\nabla f_{i}(x^{0})\| + L_{i}R\right\}+L_{\max}\nm{d^{k}}}{\max\limits_{i\in[m]\setminus\mathcal{L}}\left\{\|\nabla f_{i}(x^{0})\| + L_{i}R\right\}+\epsilon}.$$
	Since $\{x^{k}\}$ converges to a weakly Pareto point, it follows that $d^{k}\rightarrow0$. Then there exists $K>0$ such that $\|d^{k}\|\leq\frac{\epsilon}{2L_{\max}}$ for all $k\geq K$. This implies that $$\sum\limits_{i\in\mathcal{L}}\lambda^{k}_{i}\leq c_{2}~{\rm for~all}~ k\geq K,$$
	where $c_{2}:=\frac{\max\limits_{i\in[m]\setminus\mathcal{L}}\left\{\|\nabla f_{i}(x^{0})\| + L_{i}R\right\}+\frac{\epsilon}{2}}{\max\limits_{i\in[m]\setminus\mathcal{L}}\left\{\|\nabla f_{i}(x^{0})\| + L_{i}R\right\}+\epsilon}<1.$ By substituting the above bound into (\ref{e5.15}), we obtain
	$$\|x^{k+1}-x^{*}\|\leq\sqrt{1-(1-c_{2})\frac{\min\limits_{i\in[m]\setminus\mathcal{L}}\mu_{i}}{L_{\max}}}\|x^{k}-x^{*}\|~{\rm for~all}~ k\geq K.$$
	Without loss of generality, there exists $\kappa\in(0,1)$ such that $\|x^{k}-x^{*}\|\leq\kappa^{k}\|x^{0}-x^{*}\|.$
\end{proof}
\vspace{1mm}
\par From Proposition \ref{pl2}, PGMO only achieves $R$-linear convergence. The following example illustrates that $\sum_{i\in\mathcal{L}}\lambda^{k}_{i}=1$ at the early stage in the iterations. In other words, we can not obtain $Q$-linear convergence of the method for (\ref{LCMOP}). Beside, the rate of $R$-linear convergence can be very slow.
\vspace{1mm}
\begin{example}
	Consider the following multiobjective optimization problem:
	\begin{align*}
		\min\limits_{x\in \mathcal{X}}\left( f_{1}(x), f_{2}(x)\right),
	\end{align*}
	where $f_{1}(x)=\frac{1}{2}x_{1}^{2}+\frac{1}{2}x_{2}^{2}$, $f_{2}(x)=cx_{1}$ ($c$ is a relative small positive constant), and $\mathcal{X}=\{(x_{1},x_{2}):x_{1}\geq 0,x_{2}=0\}.$ By simple calculations, we have
	$$\nabla f_{1}(x)=(x_{1},x_{2})^{T},~\nabla f_{2}(x)=(c,0)^{T},~L_{\max}=1$$
	and the unique Pareto solution is $(0,0)^{T}$. Given a feasible $x^{0}$, at the early stage ($x^{k}_{1}>c$) of PGMO without line search, we have $\lambda_{2}^{k}=1$, and $d^{k}=(-c,0)$. At this stage, we have ${\nm{x^{k+1}-x^{*}}}/{\nm{x^{k}-x^{*}}}={(x^{k}_{1}-c)}/{x^{k}_{1}}$, which tends to $1$ for sufficient small $c$. 
\end{example}

\section{Scaled proximal gradient method with unknown smoothness parameters for MCOPs}\label{sec5}
In the previous section, global smoothness parameters were employed to scale each objective; however, these parameters are often unknown and can be overly conservative. This section is devoted to a scaled proximal gradient method with unknown smoothness parameters.

\subsection{Estimating the local smoothness parameters}
Firstly, we propose a backtracking method to estimate the local smoothness parameters. 
\begin{algorithm}  
	\caption{{\ttfamily{\small Backtracking}}}\label{alg3}
	\begin{algorithmic}[1]  
		\REQUIRE{$0<\alpha^{k}_{i}\leq L_{i},s_{i}=0,i\in[m],\tau>1$}
		\REPEAT{
			\STATE{Update $\alpha_{i}^{k}=\tau^{s_{i}}\alpha_{i}^{k},i\in[m]$}
			\STATE{Update $x^{k+1}:=	\mathop{\arg\min}\limits_{x\in\mathbb{R}^{n}}\max\limits_{i\in[m]}\left\{\frac{
					\left\langle\nabla f_{i}(x^{k}),x-x^{k}\right\rangle + g_{i}(x)-g_{i}(x^{k})}{\alpha^{k}_{i}}+\frac{1}{2}\|x-x^{k}\|^{2}\right\}$}
			\FOR{$i=1,\cdots,m$}
			\IF{$f_{i}(x^{k+1})-f_{i}(x^{k})>\left\langle\nabla f_{i}(x^{k}),x^{k+1}-x^{k}\right\rangle+\frac{\alpha^{k}_{i}}{2}\|x^{k+1}-x^{k}\|^{2} $}
			\STATE{Update $s_{i}=s_{i}+1$}
			\ENDIF
			\ENDFOR}
		\UNTIL{$f_{i}(x^{k+1})-f_{i}(x^{k})\leq\left\langle\nabla f_{i}(x^{k}),x^{k+1}-x^{k}\right\rangle+\frac{\alpha^{k}_{i}}{2}\|x^{k+1}-x^{k}\|^{2},~i\in[m] $}
	\end{algorithmic}
\end{algorithm}

Since $f_{i}$ is $L_{i}$-smooth for $i\in[m]$, the {\it {repeat} loop} terminates in a finite number of iterations, and $\alpha_{i}^{k}<\tau L_{i},~i\in[m]$. A notable advantage of this approach is its ability to adapt to the local smoothness based on the trajectory of the iterations. However, the use of backtracking increases the computational cost per iteration due to the repeated solving of subproblems.
\subsection{Line search}
To avoid solving a subproblem in each backtracking procedure, we apply the Armijo line search proposed by Tanabe et al. \cite{TFY2019}.
\begin{algorithm}  
	\caption{\ttfamily Armijo\_line\_search}\label{alg1} 
	\begin{algorithmic}[1]
		\REQUIRE{ $x^{k}\in\mathbb{R}^{n},d^{k}\in\mathbb{R}^{n},Jf(x^{k})\in\mathbb{R}^{m\times n},\sigma\in(0,1), t_{k}=1$}
		\WHILE{$F(x^{k}+t_{k} d^{k})- F(x^{k}) \not\preceq t_{k}\sigma  (Jf(x^{k})d^{k}+g(x^{k}+ d^{k})-g(x^{k}))$}
		\STATE{Update $t_{k}:=  {t_{k}}/{2}$ } 
		\ENDWHILE
		\RETURN{$t_{k}$}
	\end{algorithmic}
\end{algorithm}

 The scaled proximal gradient method with line search for MCOPs is described as follows.
\begin{algorithm}  
	\caption{{\ttfamily{Scaled\_proximal\_gradient\_method\_with\_line\_search\_for\_MCOPs}}}\label{alg2} 
	\begin{algorithmic}[1]
		\REQUIRE{$x^{0}\in\mathbb{R}^{n}$}
		\FOR{$k=0,\cdots$}
		\STATE{Update $\alpha^{k}\in\mathbb{R}^{m}_{++}$}
		\STATE{Update $d^{k}:=	\mathop{\arg\min}\limits_{x\in\mathbb{R}^{n}}\max\limits_{i\in[m]}\left\{\frac{
				\left\langle\nabla f_{i}(x^{k}),d\right\rangle + g_{i}(x^{k}+d)-g_{i}(x^{k})}{\alpha^{k}_{i}}+\frac{1}{2}\|d\|^{2}\right\}$}
		\IF{$d^{k}=0$}
		\RETURN{$x^{k}$  }
		\ELSE{
			\STATE{Update $t_{k}:=$ {\ttfamily Armijo\_line\_search}$\left(x^{k},d^{k},Jf(x^{k})\right)$}
			\STATE{Update $x^{k+1}:= x^{k}+t_{k}d^{k}$}}
		\ENDIF
		\ENDFOR
	\end{algorithmic}
\end{algorithm}

It is worth noting that we don't specify how to select $\alpha^{k}$ in Algorithm \ref{alg2}. The direct question arises: Does $\alpha^{k}$ affect the performance of Algorithm \ref{alg2}? In the following, we will examine the role that $\alpha^{k}$ plays in the Algorithm \ref{alg2}.

\par Firstly, we prove that $d^{k}$ is a descent direction for $F$. 
\vspace{1mm}
\begin{lemma}
	For the $d^{k}$ in Algorithm \ref{alg2}, we have
	\begin{equation}\label{E6}
		\left\langle\nabla f_{i}(x^{k}),d^{k}\right\rangle + g_{i}(x^{k}+ d^{k})-g_{i}(x^{k})\leq-\alpha_{i}^{k}\|d^{k}\|^{2}~{\rm for~all}~ i\in[m].
	\end{equation}
\end{lemma}
\begin{proof}
	The assertion can be obtained by using the same arguments as in the proof of \cite[Lemma 4.1]{TFY2019}.
\end{proof}

Next, we give a lower bound of stepsize in each iteration. 
\vspace{1mm}
\begin{lemma}
	Assume $f_{i}$ is $L_{i}$-smooth for $i\in[m]$, then the $k$-th stepsize generated by Algorithm \ref{alg1} satisfies $t_{k}\geq t_{k}^{\min}:=\min\left\{\min_{i\in[m]}\{{(1-\sigma)\alpha_{i}^{k}}/{L_{i}}\},1\right\}$.
\end{lemma}
\vspace{1mm}
\begin{proof} It is sufficient to prove $t_{k}<1$, then backtracking is conducted, leading to the inequality:
	\begin{equation}\label{E4.4}
		F_{i}\left(x^{k}+2{t_{k}}d^{k}\right)-F_{i}(x^{k})>2\sigma{t_{k}}(\left\langle\nabla f_{i}(x^{k}),d^{k}\right\rangle + g_{i}(x^{k}+ d^{k})-g_{i}(x^{k}))
	\end{equation}
	for some $i\in[m]$. Since $f_{i}$ is $L_{i}$-smooth for $i\in[m]$, we can derive the following inequalities:
	\begin{equation}\label{up}
		\begin{aligned}
			&~~~F_{i}\left(x^{k}+2{t_{k}}d^{k}\right)-F_{i}(x^{k})\\
			&\leq2{t_{k}}\left\langle\nabla f_{i}(x^{k}),d^{k}\right\rangle + g_{i}(x^{k}+2t_{k}d^{k}) - g_{i}(x^{k})+ \frac{L_{i}}{2}\left\|2t_{k}d^{k}\right\|^{2}\\
			&\leq2t_{k}(\left\langle\nabla f_{i}(x^{k}),d^{k}\right\rangle + g_{i}(x^{k}+ d^{k})-g_{i}(x^{k}))+\frac{L_{i}}{2}\left\|2t_{k}d^{k}\right\|^{2},
		\end{aligned}
	\end{equation}
	where the second inequality follows from the convexity of $g_{i}$ and the fact that $2t_{k}\in(0,1]$.  Combining this inequality with (\ref{E4.4}), we obtain
	$$(\sigma - 1)(\left\langle\nabla f_{i}(x^{k}),d^{k}\right\rangle + g_{i}(x^{k}+ d^{k})-g_{i}(x^{k}))\leq{L_{i}t_{k}}\left\|d^{k}\right\|^{2}$$
	for some $i\in[m]$. Utilizing (\ref{E6}), we arrive at
	\begin{equation}\label{et}
		t_{k}\geq\frac{(1-\sigma)\alpha^{k}_{i}}{L_{i}}
	\end{equation}
	for some $i\in[m]$, it holds that $t_{k}\geq t_{\min}$. This completes the proof.
\end{proof}

Before presenting the convergence analysis of Algorithm \ref{alg2}, we introduce two types of merit functions for (\ref{MCOP}) that quantify the gap between the current point and the optimal solution. 
\begin{equation}\label{u}
	u_{0}^{\alpha}(x):=\sup\limits_{y\in\mathbb{R}^{n}}\min\limits_{i\in[m]}\left\{\frac{F_{i}(x)-F_{i}(y)}{\alpha_{i}}\right\},
\end{equation}

\begin{equation}\label{v}
	w_{\ell}^{\alpha}(x):=\max\limits_{y\in\mathbb{R}^{n}}\min\limits_{i\in[m]}\left\{\frac{
		\left\langle\nabla f_{i}(x),x-y\right\rangle + g_{i}(x)-g_{i}(y)}{\alpha_{i}}-\frac{\ell}{2}\|x-y\|^{2}\right\},
\end{equation}
where $\alpha\in\mathbb{R}^{m}_{++}$, $\ell>0$.

\par We can demonstrate that $u_{0}^{\alpha}$ and $w_{\ell}^{\alpha}$ serve as merit functions, satisfying the criteria of weak Pareto and critical points, respectively.
\vspace{1mm}
\begin{proposition}
	Let $u_{0}^{\alpha}$ and $w_{\ell}^{\alpha}$ be defined as {\rm(\ref{u})} and {\rm(\ref{v})}, respectively. Then, for all $\alpha\in\mathbb{R}^{m}_{++}$ and $\ell>0$ the following statements hold.
	\begin{itemize}
		\item[$\mathrm{(i)}$]  $x\in\mathbb{R}^{n}$ is a weak Pareto solution of {\rm(\ref{MCOP})} if and only if $u_{0}^{\alpha}(x)=0$.
		\item[$\mathrm{(ii)}$]  $x\in\mathbb{R}^{n}$ is a Pareto critical point of {\rm(\ref{MCOP})} if and only if $w_{\ell}^{\alpha}(x)=0$.
	\end{itemize}
\end{proposition}
\vspace{1mm}
\begin{proof}
	The assertions (i) and (ii) can be obtained by using similar arguments as in the proofs of \cite[Theorem 3.1]{TFY2020} and \cite[Theorem 3.9]{TFY2020}, respectively.
\end{proof}
\vspace{1mm}
\begin{proposition}\label{p2}
	Let $w_{\ell}^{\alpha}$ be defined as  {\rm(\ref{v})}. If $0<\ell\leq r$, then
	$$w_{r}^{\alpha}(x)\leq w_{\ell}^{\alpha}(x)\leq\frac{r}{\ell}w_{r}^{\alpha}(x)~{\it for~all}~ x\in\mathbb{R}^{n}. $$
\end{proposition}

\begin{proof}
	The assertion can be obtained by using the similar arguments as in the proof of \cite[Theorem 4.2]{TFY2020}.
\end{proof}

We are now in a position to establish the linear convergence of Algorithm \ref{alg2}.

\begin{theorem}\label{t4}
	Assume that $f_{i}$ is $L_{i}$-smooth and strongly convex with modulus $\mu_{i}>0$, for $i\in[m]$. Let $\{x^{k}\}$ be the sequence generated by Algorithm \ref{alg2}. Then, we have
	$$u_{0}^{\alpha^{k}}(x^{k+1})\leq\left(1-\sigma t^{\min}_{k}r_{k}\right)u_{0}^{\alpha^{k}}(x^{k})$$
	for all $k\geq0$, where $r_{k}:=\min\{\min_{i\in[m]}\{\mu_{i}/\alpha^{k}_{i}\},1\}$.
\end{theorem}
\vspace{1mm}
\begin{proof}
	The Armijo line search satisfies
	$$F_{i}(x^{k+1})-F_{i}(x^{k})\leq t_{k}\sigma(\left\langle\nabla f_{i}(x^{k}),d^{k}\right\rangle+g_{i}(x^{k}+d^{k})-g_{i}(x^{k})).$$
	A direct calculation gives
	\begin{equation}\label{E15}
		\begin{aligned}
			&~~~\frac{F_{i}(x^{k+1})-F_{i}(x^{k})}{\alpha^{k}_{i}}\\
			&\leq t_{k}\sigma\left(\frac{\left\langle\nabla f_{i}(x^{k}),d^{k}\right\rangle+g_{i}(x^{k}+d^{k})-g_{i}(x^{k})}{\alpha^{k}_{i}}+\frac{1}{2}\|d^{k}\|^{2}\right)\\
			&\leq t_{k}\sigma\max\limits_{i\in[m]}\left\{\frac{\left\langle\nabla f_{i}(x^{k}),d^{k}\right\rangle+g_{i}(x^{k}+d^{k})-g_{i}(x^{k})}{\alpha^{k}_{i}}+\frac{1}{2}\|d^{k}\|^{2}\right\}\\
			&=-t_{k}\sigma w_{1}^{\alpha^{k}}(x^{k})\\
			&\leq -t^{\min}_{k}\sigma w_{1}^{\alpha^{k}}(x^{k}).
		\end{aligned}
	\end{equation}
	Furthermore, due to the strong convexity of $f_{i}$ for $i\in[m]$, we have
	\begin{equation}\label{low}
		\begin{aligned}&~~~
			\frac{F_{i}(x^{k})-F_{i}(x)}{\alpha^{k}_{i}}\\&\leq\frac{\left\langle\nabla f_{i}(x^{k}),x^{k}-x\right\rangle+g_{i}(x^{k})-g_{i}(x)}{\alpha^{k}_{i}}-\frac{\min_{i\in[m]}\{\mu_{i}/\alpha^{k}_{i}\}}{2}\|x-x^{k}\|^{2}
		\end{aligned}
	\end{equation}
	for all $i\in[m]$ and $x\in\mathbb{R}^{n}$.
	Taking the supremum and minimum  with respect to $x\in\mathbb{R}^{n}$ and $i\in [m]$ on both sides, respectively, we obtain
	\begin{align*}
		&~~~~\sup\limits_{x\in\mathbb{R}^{n}}\min\limits_{i\in[m]}\left\{\frac{F_{i}(x^{k})-F_{i}(x)}{\alpha^{k}_{i}}\right\}\\
		&\leq\sup\limits_{x\in\mathbb{R}^{n}}\min\limits_{i\in[m]}\left\{\frac{\left\langle\nabla f_{i}(x^{k}),x^{k}-x\right\rangle+g_{i}(x^{k})-g_{i}(x)}{\alpha^{k}_{i}}-\frac{\min_{i\in[m]}\{\mu_{i}/\alpha^{k}_{i}\}}{2}\|x-x^{k}\|^{2}\right\}\\
		&\leq\frac{1}{r_{k}}\sup\limits_{x\in\mathbb{R}^{n}}\min\limits_{i\in[m]}\left\{\frac{\left\langle\nabla f_{i}(x^{k}),x^{k}-x\right\rangle+g_{i}(x^{k})-g_{i}(x)}{\alpha^{k}_{i}}-\frac{1}{2}\|x-x^{k}\|^{2}\right\},
	\end{align*}
	where the second inequality is due to Proposition \ref{p2}. The above inequalities can be rewritten as
	$$u_{0}^{\alpha^{k}}(x^{k})\leq\frac{1}{r_{k}}w_{1}^{\alpha^{k}}(x^{k}).$$
	Together with (\ref{E15}), the preceding inequality yields
	$$\frac{F_{i}(x^{k+1})-F_{i}(x^{k})}{\alpha^{k}_{i}}\leq- \sigma t^{\min}_{k}r_{k} u_{0}^{\alpha^{k}}(x^{k}).$$
	Then, for all $x\in\mathbb{R}^{n}$, we have
	$$\frac{F_{i}(x^{k+1})-F_{i}(x)}{\alpha^{k}_{i}}\leq\frac{F_{i}(x^{k})-F_{i}(x)}{\alpha^{k}_{i}} - \sigma t^{\min}_{k}r_{k} u_{0}^{\alpha^{k}}(x^{k}).$$
	Taking the supremum and minimum  with respect to $x\in\mathbb{R}^{n}$ and $i\in [m]$ on both sides, respectively, we obtain
	$$u_{0}^{\alpha^{k}}(x^{k+1})\leq\left(1- \sigma t^{\min}_{k}r_{k}\right)u_{0}^{\alpha^{k}}(x^{k}).$$
	This completes the proof.
\end{proof} 

The following theorem shows that $\alpha^{k}\in\mathbb{R}^{m}_{++}$ plays a significant role in the convergence rate of Algorithm \ref{alg2}.
\begin{theorem}\label{t5.5}
	Assume that $f_{i}$ is $L_{i}$-smooth and strongly convex with modulus $\mu_{i}>0$, for $i\in[m]$. Let $\{x^{k}\}$ be the sequence generated by Algorithm \ref{alg2}. Then, the following statements hold.
	\begin{itemize}
		\item[$\mathrm{(i)}$] For any $\alpha^{k}\in\mathbb{R}_{++}$, we have
		$$(1-\sigma)\min\left\{\frac{\alpha^{k}_{\min}\mu_{\min}}{\alpha^{k}_{\max}L_{\max}}, \frac{\alpha^{k}_{\min}}{L_{\max}},\frac{\mu_{\min}}{\alpha^{k}_{\max}}\right\}\leq t^{\min}_{k}r_{k}\leq(1-\sigma)\min\limits_{i\in[m]}\left\{\frac{\mu_{i}}{L_{i}}\right\},$$
		where $\alpha_{\min}^{k}:=\min\{\alpha_{i}^{k},i\in[m]\}$ and $\alpha_{\max}^{k}:=\max\{\alpha_{i}^{k},i\in[m]\}$.
		\item[$\mathrm{(ii)}$] If $\alpha^{k}_{i}=\ell>0$ for $i\in[m]$, we have
		$t^{\min}_{k}r_{k}=(1-\sigma)\min\left\{\frac{\mu_{\min}}{L_{\max}}, \frac{\ell}{L_{\max}},\frac{\mu_{\min}}{\ell}\right\}$, then
		$$u_{0}(x^{k+1})\leq\left(1- \sigma(1-\sigma)\min\left\{\frac{\mu_{\min}}{L_{\max}}, \frac{\ell}{L_{\max}},\frac{\mu_{\min}}{\ell}\right\} \right)u_{0}(x^{k}).$$
		\item[$\mathrm{(iii)}$] If $\alpha^{k}_{i}=\mu_{i}$ for $i\in[m]$ or $\alpha^{k}_{i}=L_{i}$ for $i\in[m]$, we have $t^{\min}_{k}r_{k}=(1-\sigma)\min_{i\in[m]}{\mu_{i}}/{L_{i}}$, then
		$$u_{0}^{\alpha^{k}}(x^{k+1})\leq\left(1- \sigma (1-\sigma)\min\limits_{i\in[m]}\left\{\frac{\mu_{i}}{L_{i}}\right\}\right)u_{0}^{\alpha^{k}}(x^{k}).$$
	\end{itemize}
\end{theorem}
\begin{proof}
	(i) Firstly, we prove the right hand side of the inequality. With loss of generality, let us assume that
	$$\frac{\mu_{1}}{L_{1}}=\min\limits_{i\in[m]}\left\{\frac{\mu_{i}}{L_{i}}\right\}.$$
	Recall the definition of $t^{\min}_{k}$ and $r_{k}$, we have
	$t^{\min}_{k}\leq(1-\sigma){\alpha^{k}_{1}}/{L_{1}}$ and $r_{k}\leq {\mu_{1}}/{\alpha^{k}_{1}}$, it follows that
	$$t^{\min}_{k}r_{k}\leq(1-\sigma)\frac{\mu_{1}}{L_{1}}=(1-\sigma)\min\limits_{i\in[m]}\left\{\frac{\mu_{i}}{L_{i}}\right\}.$$
	For the left hand side of the inequality, we distinguish three cases: (a) $\alpha_{\max}^{k}\leq\mu_{\min}$, (b) $L_{\max}\leq\alpha^{k}_{\min}$ and (c) $\mu_{\min}<\alpha_{\max}^{k}$ and $\alpha^{k}_{\min}<L_{\max}$. If (a) holds, we have $r_{k}=1$ and $t_{k}^{\min}\geq(1-\sigma){\alpha^{k}_{\min}}/{L_{\max}}$, it follows that
	$$t_{k}^{\min}r_{k}\geq(1-\sigma)\frac{\alpha^{k}_{\min}}{L_{\max}}.$$
	If (b) holds, we derive that $r_{k}\geq \mu_{\min}/\alpha^{k}_{\max}$ and $t^{\min}_{k}\geq(1-\sigma)$, therefore,
	$$t_{k}^{\min}r_{k}\geq(1-\sigma)\mu_{\min}/\alpha^{k}_{\max}.$$
	If (c) holds, we deduce that $r_{k}> \mu_{\min}/\alpha^{k}_{\max}$ and $t^{\min}_{k}>(1-\sigma)\alpha^{k}_{\min}/L_{\max}$, then
	$$t_{k}^{\min}r_{k}\geq(1-\sigma)\frac{\alpha^{k}_{\min}\mu_{\min}}{\alpha^{k}_{\max}L_{\max}}.$$
	Consequently, the left hand side of the inequality in (i) holds.
	\par (ii) By distinguishing three cases as in the proof of (i) with $\alpha_{\max}^{k}=\alpha_{\max}^{k}=\ell$, the desire result follows. 
	\par (iii) If $\alpha^{k}_{i}=\mu_{i}$ for $i\in[m]$, we have $t_{k}^{\min}=(1-\sigma)\min_{i\in[m]}\{\mu_{i}/L_{i}\}$ and $r_{k}=1$. While $\alpha^{k}_{i}=L_{i}$ for $i\in[m]$, we have $t_{k}^{\min}=(1-\sigma)$ and $r_{k}=\min_{i\in[m]}\{\mu_{i}/L_{i}\}$. In both cases, we have 
	$$t^{\min}_{k}r_{k}=(1-\sigma)\min_{i\in[m]}\left\{\frac{\mu_{i}}{L_{i}}\right\}.$$
	The desire result follows.
\end{proof}
\begin{remark}\label{r5.6}
	Theorem \ref{t5.5} demonstrates that the scaling parameters $\{\alpha^{k}_{i}:i\in[m]\}$ are pivotal in determining the linear convergence rate of Algorithm \ref{alg2}. The primary conclusions can be summarized as follows. 
	\begin{itemize}
		\item[$\mathrm{(i)}$] If $\alpha^{k}_{i}=\ell>0$ for $i\in[m]$, then the SPGMO with line search reduces to the PGMO with line search \cite{TFY2019}. Consequently, Theorem \ref{t5.5} (ii) indicates that the PGMO with line search is unable to address objective imbalances.
		\item[$\mathrm{(ii)}$] By setting $\alpha^{k}_{i}=\mu_{i}$ for $i\in[m]$ or $\alpha^{k}_{i}=L_{i}$ for $i\in[m]$, we derive the objective imbalances parameter $\zeta=1$ for scaled problems $F^{L}$ and $F^{\mu}$, thereby completely mitigating objective imbalances. In these scenarios, the SPGMO with line search achieves optimal linear convergence and enjoys favorable performance for well-conditioned problems.
		\item[$\mathrm{(iii)}$] The convergence rate in Theorem \ref{t5.5} (iii) is attributable to the global upper and lower bounds employed in (\ref{up}) and (\ref{low}), respectively. As a result, setting $\mu_{i}\leq\alpha_i^k\leq L_{i}$ leverages the problem's local geometry, whereas $\alpha^{k}_{i}=\mu_{i}$ or $\alpha^{k}_{i}=L_{i}$ is too conservative as the curvature of $f_{i}$ can be pretty different.  Intuitively, the linear convergence rate for Algorithm \ref{alg2} with $\mu_{i}\leq\alpha_i^k\leq L_{i}$ for $i\in[m]$ is, roughly speaking, not less than $1-\sigma(1-\sigma)\min_{i\in[m]}\left\{{\mu_{i}}/{L_{i}}\right\}$.
	\end{itemize}
\end{remark}

\section{Scaled proximal method with Nesterov's acceleration for MCOPs}\label{sec6} 
Tanabe et al. \cite{TFY2022} proposed an accelerated proximal gradient method for MCOPs (APGMO), and established the well-known $\mathcal{O}(1/k^{2})$ for convex scenarios. The crux of the APGMO lies in the subproblem:
\begin{equation}\label{adk1}
	\min\limits_{x\in\mathbb{R}^{n}}\max\limits_{i\in[m]}\left\{{
		\left\langle\nabla f_{i}(y^{k}),x-y^{k}\right\rangle + g_{i}(x)+f_{i}(y^{k})-F_{i}(x^{k})}+\frac{L_{\max}}{2}\|x-y^{k}\|^{2}\right\},
\end{equation}
where the inside of max operator approximates $F_{i}(x)-F_{i}(y^{k})$ rather that $F_{i}(x)-F_{i}(x^{k})$, which deeply affects the proof in the multiobjective case. We refer the reader to \cite{TFY2022} for more details. It is worth noting that the linear convergence of APGMO for strongly convex cases was not analyzed in \cite{TFY2022}. This raises an intriguing question: Does the APGMO with an appropriate momentum parameter exhibit the linear convergence $\mathcal{O}((1-\sqrt{{\mu_{\min}}/{L_{\max}}})^{k})$ for strongly convex scenarios? 
\subsection{Accelerated scaled proximal gradient method for MCOPs} In this subsection, we propose an accelerated scaled proximal method that alleviate the objective imbalances. Inspired by the work of Tanabe et al. \cite{TFY2022}, the subproblem is defined as follows:
\begin{equation}\label{adk}
	\min\limits_{x\in\mathbb{R}^{n}}\max\limits_{i\in[m]}\left\{\frac{
		\left\langle\nabla f_{i}(y^{k}),x-y^{k}\right\rangle + g_{i}(x)+f_{i}(y^{k})-F_{i}(x^{k})}{L_{i}}+\frac{1}{2}\|x-y^{k}\|^{2}\right\}.
\end{equation}
Assume that $f_{i}$ is strongly convex with module $\mu_{i}\geq0$ and $L_{i}$-smooth for $i\in[m]$, we define 
$$\hat{\mu}:=\min_{i\in[m]}\left\{\frac{\mu_{i}}{L_{i}}\right\}.$$
It is evident that $0\leq\hat{\mu}\leq1$. The accelerated scaled proximal gradient method for MCOPs (ASPGMO) is described as follows.
\begin{algorithm}  
	\caption{{\ttfamily{Accelerated\_scale\_proximal\_gradient\_method\_for\_MCOPs}}}\label{apgmo}
	\begin{algorithmic}[1] 
		\REQUIRE{$x^{-1}=x^{0}\in\mathbb{R}^{n}$}
		\FOR{$k=0,\cdots$}
		\STATE{$\gamma_{k}=\frac{(\theta_{k}-\hat{\mu})(1-\theta_{k-1})}{(1-\hat{\mu})\theta_{k-1}}$}
		\STATE{$y^{k}=x^{k}+\gamma_{k}(x^{k}-x^{k-1})$}
		\STATE{Compute $x^{k+1}$ by solving subproblem (\ref{adk})}
		\IF{$x^{k+1}=y^{k}$}
		\RETURN{$x^{k+1}$  } 
		\ENDIF
		\ENDFOR
	\end{algorithmic}
\end{algorithm}
\begin{remark}
	Algorithm \ref{apgmo} unifies the convex and strongly convex cases by  employing distinct values of $\theta_{k}$ in the momentum parameter. When the smooth parameters are unknown, we can update the scaling parameters through backtracking.
\end{remark}
\subsection{Convergence analysis of ASPGMO}
Before proceeding with the complexity analysis, we define the following auxiliary sequences:
$$\sigma_{k}^{L}(z):=\min\limits_{i\in[m]}\left\{\frac{F_{i}(x^{k})-F_{i}(z)}{L_{i}}\right\},$$
$$\rho_{k}(z):=\nm{\frac{1}{\theta_{k}}x^{k+1}-\frac{1-\theta_{k}}{\theta_{k}}x^{k}-z}^{2}.$$
Next, we present two lemmas that will be utilized in the complexity analysis.
\vspace{1mm}
\begin{lemma}\label{l5.2}
	Assume that $f_{i}$ is strongly convex with modulus $\mu_{i}\geq0$ and $L_{i}$-smooth for $i\in[m]$.  Let $\{x^{k}\}$ and $\{y^{k}\}$ be sequences generated by Algorithm \ref{apgmo}. Then, we have
	\begin{equation}\label{s1}
		\sigma_{k+1}^{L}(z)\leq \dual{x^{k+1}-y^{k},z-y^{k}}-\frac{\hat{\mu}}{2}\nm{y^{k}-z}^{2}-\frac{1}{2}\nm{x^{k+1}-y^{k}}^{2},
	\end{equation}
	and
	\begin{equation}\label{s2}
		\sigma_{k+1}^{L}(z)-\sigma_{k}^{L}(z)\leq \dual{x^{k+1}-y^{k},x^{k}-y^{k}}-\frac{1}{2}\nm{x^{k+1}-y^{k}}^{2}
	\end{equation}
	for all $z\in\mathbb{R}^{n}$ and $k\geq0$.
\end{lemma}
\begin{proof}
	By using strong convexity in (\ref{s1}), the assertions can be obtained by using similar arguments as in the proof of  \cite[Lemma 7]{TFY2022}.  
\end{proof}
\vspace{1mm}
\begin{lemma}[See Theorem 9 of \cite{TFY2022}]
	Let $\{x^{k}\}$ be sequence generated by Algorithm \ref{apgmo}. Then, we have $x^{k}\in\mathcal{L}_{F}(F(x^{0}))$ for all $k\geq0$.
\end{lemma}
\vspace{1mm}

Using Lemma \ref{l5.2}, we establish an auxiliary inequality as follows.
\vspace{1mm}
\begin{lemma}
	Assume that $f_{i}$ is strongly convex with modulus $\mu_{i}\geq0$ and $L_{i}$-smooth for $i\in[m]$.  Let $\{x^{k}\}$ and $\{y^{k}\}$ be sequences generated by Algorithm \ref{apgmo} with $\theta_{k}\geq\hat{\mu}$ for all $k\geq0$. Then, we have
	\begin{equation}\label{s3}
		\sigma_{k+1}^{L}(z)+\frac{\theta^{2}_{k}}{2}\rho_{k}(z)\leq(1-\theta_{k})\sigma_{k}^{L}(z)+\frac{\theta_{k}(\theta_{k}-\hat{\mu})}{2}\rho_{k-1}(z)
	\end{equation}
	for all $z\in\mathbb{R}^{n}$ and $k\geq0$.
\end{lemma}
\begin{proof}
	Multiplying inequality (\ref{s1}) by $\theta_{k}$ and inequality (\ref{s2}) by $(1-\theta_{k})$ and adding them together, we have
	\begin{equation}\label{e5.5}
		\begin{aligned}
			&\sigma_{k+1}^{L}(z) - (1-\theta_{k})\sigma_{k}^{L}(z)\\ &\leq -\frac{1}{2}\nm{x^{k+1}-y^{k}}^{2} -\frac{\theta_{k}\hat{\mu}}{2}\nm{y^{k}-z}^{2} + \dual{x^{k+1}-y^{k},(1-\theta_{k})x^{k}+\theta_{k}z-y^{k}}\\
			&=-\frac{1}{2}\nm{x^{k+1}-y^{k}}^{2} -\frac{\theta_{k}\hat{\mu}}{2}\nm{y^{k}-z}^{2} + \frac{1}{2}(\nm{x^{k+1}-y^{k}}^{2}+\\
			&~~~~~~~\nm{(1-\theta_{k})x^{k}+\theta_{k}z-y^{k}}^{2}-\nm{x^{k+1}-(1-\theta_{k})x^{k}-\theta_{k}z}^{2})\\
			&= -\frac{\theta_{k}\hat{\mu}}{2}\nm{y^{k}-z}^{2}+\frac{\theta_{k}^{2}}{2}\left(\nm{\frac{1-\theta_{k}}{\theta_{k}}x^{k}+z-\frac{1}{\theta_{k}}y^{k}}^{2}-\rho_{k}(z)\right),
		\end{aligned}
	\end{equation}
	where the first inequality is due to the relation $2\dual{a,b}\leq\nm{a}^{2}+\nm{b}^{2}-\nm{a-b}^{2}$ with $a:=x^{k+1}-y^{k}$ and $b:=(1-\theta_{k})x^{k}+\theta_{k}z-y^{k}$. By reorganizing the terms in $\frac{1-\theta_{k}}{\theta_{k}}x^{k}+z-\frac{1}{\theta_{k}}y^{k}$ carefully, we obtain
	\begin{align*}
		&\frac{\theta_{k}^{2}}{2}\nm{\frac{1-\theta_{k}}{\theta_{k}}x^{k}+z-\frac{1}{\theta_{k}}y^{k}}^{2}\\
		&~~~=\frac{\theta_{k}^{2}}{2}\nm{\frac{\hat{\mu}}{\theta_{k}}(z-y^{k})+\frac{\theta_{k}-\hat{\mu}}{\theta_{k}}\left(z+\frac{1-\theta_{k}}{\theta_{k}-\hat{\mu}}x^{k}-\frac{1-\hat{\mu}}{\theta_{k}-\hat{\mu}}y^{k}\right)}^{2}\\
		&~~~\leq \frac{\theta_{k}\hat{\mu}}{2}\nm{z-y^{k}}^{2}+\frac{\theta_{k}(\theta_{k}-\hat{\mu})}{2}\nm{z+\frac{1-\theta_{k}}{\theta_{k}-\hat{\mu}}x^{k}-\frac{1-\hat{\mu}}{\theta_{k}-\hat{\mu}}y^{k}}^{2}\\
		&~~~=\frac{\theta_{k}\hat{\mu}}{2}\nm{z-y^{k}}^{2}+\frac{\theta_{k}(\theta_{k}-\hat{\mu})}{2}\rho_{k-1}(z),
	\end{align*}
	where the inequality is due to the facts that $\theta_{k}>\hat{\mu}$ and the convexity of $\nm{\cdot}^{2}$, and the last equality follows by the definition of $y^{k}$ in Algorithm \ref{apgmo}.
	Substituting the above inequality into (\ref{e5.5}), we derive
	$$\sigma_{k+1}^{L}(z) - (1-\theta_{k})\sigma_{k}^{L}(z)\leq \frac{\theta_{k}(\theta_{k}-\hat{\mu})}{2}\rho_{k-1}(z)-\frac{\theta_{k}^{2}}{2}\rho_{k}(z).$$
	Hence, the inequality (\ref{s3}) holds for all $z\in\mathbb{R}^{n}$ and $k\geq0$.
\end{proof}
\vspace{1mm}
\par We define a Lyapunov function for $\mu_{i}=0,~i\in[m]$:
\begin{equation}\label{ly1}
	\mathcal{E}_{k+1}(z):=\frac{\sigma_{k+1}^{L}(z)}{\theta^{2}_{k}}+\frac{1}{2}\rho_{k}(z).
\end{equation}
We will show that $\mathcal{E}_{k+1}(z)\leq \mathcal{E}_{k}(z)$ for all $k\geq0$ and establish the convergence rate in the following theorem.
\begin{theorem}
	Assume that $f_{i}$ is convex and $L_{i}$-smooth for $i\in[m]$. Let $\theta_{k}={2}/({k+2})$ in Algorithm \ref{apgmo}, i.e., $\gamma_{k}=({k-1})/({k+2})$. Then, we have
	$$\mathcal{E}_{k+1}(z)\leq\mathcal{E}_{k}(z)$$
	for all $k\geq0$. If the level set $\mathcal{L}_{F}(F(x^{0}))$ is bounded, then
	$$u_{0}^{L}(x^{k})\leq\frac{2R^{2}}{(k+1)^{2}}$$
	for all $k\geq0$, where $R:=\max\left\{\nm{x-y}:x,y\in \mathcal{L}_{F}(x^{0})\right\}$.
\end{theorem}
\begin{proof}
	In this case, we have $\hat{\mu}=0$, thus $\theta_{k}\geq\hat{\mu}$ holds for all $k\geq0$. By substituting $\hat{\mu}=0$ into (\ref{s3}), it follows that
	$$\sigma_{k+1}^{L}(z)+\frac{\theta^{2}_{k}}{2}\rho_{k}(z)\leq(1-\theta_{k})\sigma_{k}^{L}(z)+\frac{\theta_{k}^{2}}{2}\rho_{k-1}(z).$$
	Dividing both sides of the above inequality by $\theta_{k}^{2}$, we have
	$$\frac{\sigma_{k+1}^{L}(z)}{\theta^{2}_{k}}+\frac{1}{2}\rho_{k}(z)\leq\frac{1-\theta_{k}}{\theta^{2}_{k}}\sigma_{k}^{L}(z)+\frac{1}{2}\rho_{k-1}(z).$$
	Recall that $\theta_{k}=\frac{2}{k+2}$, then we can deduce 
	$\frac{1-\theta_{k}}{\theta_{k}^{2}}\leq\frac{1}{\theta_{k-1}^{2}}$. This together with the above inequality implies 
	$$\mathcal{E}_{k+1}(z)\leq\mathcal{E}_{k}(z).$$
	Notice that $\rho_{k-1}(z)\geq0$ and $x^{0}=x^{-1}$, the above inequality implies
	$$\sigma_{k}^{L}(z)\leq\theta_{k-1}^{2}\left(\frac{\sigma_{0}^{L}(z)}{\theta_{-1}^{2}}+\frac{1}{2}\nm{x^{0}-z}^{2}\right).$$
	Setting $1/\theta_{-1}^{2}=0$, and selecting $x^{k}_{*}\in\arg\max_{z\in\mathbb{R}^{n}}\min_{i\in[m]}\left\{({F_{i}(x^{k})-F_{i}(z)})/{L_{i}}\right\}$, the previous inequality implies
	$$u_{k}^{L}(x^{k})\leq\frac{2\nm{x^{0}-x^{k}_{*}}^{2}}{(k+1)^{2}}.$$
	Then, the desired result follows by $F(x^{k}_{*})\preceq F(x^{k})$ and $x^{k}\in\mathcal{L}_{F}(F(x^{0}))$.
\end{proof}

\par We also define a Lyapunov function for $\mu_{i}>0,~i\in[m]$:
\begin{equation}\label{ly2}
	\mathcal{E}^{\mu}_{k+1}(z):=\frac{1}{(1-\sqrt{\hat{\mu}})^{k+1}}\left({\sigma_{k+1}^{L}(z)}+\frac{\hat{\mu}}{2}\rho_{k}(z)\right).
\end{equation}
We will show that $\mathcal{E}_{k+1}^{\mu}(z)\leq \mathcal{E}_{k}^{\mu}(z)$ for all $k\geq0$ and establish the convergence rate in the following theorem.
\begin{theorem}
	Assume that $f_{i}$ is strongly convex with module $\mu_{i}>0$ and $L_{i}$-smooth for $i\in[m]$. Let $\theta_{k}=\sqrt{\hat{\mu}}$ in Algorithm \ref{apgmo}, i.e., $\gamma_{k}=({1-\sqrt{\hat{\mu}}})/({1+\sqrt{\hat{\mu}}})$. Then, we have
	$$\mathcal{E}_{k+1}^{\mu}(z)\leq \mathcal{E}_{k}^{\mu}(z),$$
	and
	$$u^{L}_{0}(x^{k})\leq \left(1-\sqrt{\hat{\mu}}\right)^{k}\left(u^{L}_{0}(x^{0})+\frac{\hat{\mu}}{2}R^{2}\right)$$
	for all $k\geq0$, where $R:=\max\left\{\nm{x-y}:x,y\in \mathcal{L}_{F}(F(x^{0}))\right\}$.
\end{theorem}
\begin{proof}
	Since $\theta_{k}=\sqrt{\hat{\mu}}$ and $\hat{\mu}\leq1$, thus $\theta_{k}\geq\hat{\mu}$ holds for all $k\geq0$. By substituting $\theta_{k}=\sqrt{\hat{\mu}}$ into (\ref{s3}), it follows that
	$$\sigma_{k+1}^{L}(z)+\frac{\hat{\mu}}{2}\rho_{k}(z)\leq(1-\sqrt{\hat{\mu}})\sigma_{k}^{L}(z)+\frac{\sqrt{\hat{\mu}}(\sqrt{\hat{\mu}}-\hat{\mu})}{2}\rho_{k-1}(z).$$
	Dividing both sides of the above inequality by $(1-\sqrt{\hat{\mu}})^{k+1}$, we have
	$$\mathcal{E}_{k+1}^{\mu}(z)\leq \mathcal{E}_{k}^{\mu}(z).$$
	Notice that $\rho_{k-1}(z)\geq0$ and $x^{0}=x^{-1}$, the above inequality implies
	$$\sigma_{k}^{L}(z)\leq\left(1-\sqrt{\hat{\mu}}\right)^{k}\left(\sigma_{0}^{L}(z)+\frac{\hat{\mu}}{2}\nm{x^{0}-z}^{2}\right).$$
	Selecting $x^{k}_{*}\in\arg\max_{z\in\mathbb{R}^{n}}\min_{i\in[m]}\left\{({F_{i}(x^{k})-F_{i}(z)})/{L_{i}}\right\}$, the previous inequality implies
	\begin{align*}
		u_{0}^{L}(x^{k})&\leq\left(1-\sqrt{\hat{\mu}}\right)^{k}\left(\sigma_{0}^{L}(x^{*})+\frac{\hat{\mu}}{2}\nm{x^{0}-x^{k}_{*}}^{2}\right)\\
		&\leq\left(1-\sqrt{\hat{\mu}}\right)^{k}\left(u_{0}^{L}(x^{0})+\frac{\hat{\mu}}{2}\nm{x^{0}-x^{k}_{*}}^{2}\right),
	\end{align*}
	where the second inequality is due to the definition of $u_{0}^{L}$. Notice that $f_{i}$ is strongly convex and $g_{i}$ is convex for $i\in[m]$, we conclude that $\mathcal{L}_{F}(x^{0})$ is bounded. Then, the desired result follows by the facts that $F(x^{k}_{*})\preceq F(x^{k})$ and $x^{k}\in\mathcal{L}_{F}(x^{0})$.
\end{proof}

\begin{remark}
	By setting $\gamma_{k}=({1-\sqrt{\mu_{\min}/L_{\max}}})/({1+\sqrt{\mu_{\min}/L_{\max}}})$ in APGMO, the method demonstrates a linear convergence rate of $\mathcal{O}((1-\sqrt{{\mu_{\min}}/{L_{\max}}})^{k})$ for strongly convex cases.
\end{remark}
\section{Numerical experiments}\label{sec7} 
In this section, we present numerical results to illustrate the performance of SPGMO across various problems. All numerical experiments were implemented in Python 3.7 and executed on a personal computer with an Intel Core i7-11390H, 3.40 GHz processor, and 16 GB of RAM. For the SPGMO with line search, we selected scaling parameters using the Barzilai-Borwein method \cite{CTY2023}. The subproblems of the tested algorithms were solved via dual problems (see \cite{TFY2022} for more details), which can be efficiently solved by Frank-Wolfe method. We set $\sigma=10^{-4}$ in the line search procedure and employed stopping criteria of $\|d^{k}\|\leq 10^{-4}$ or $\nm{x^{k+1}-y^{k}}\leq10^{-4}$ for all tested algorithms to ensure termination after a finite number of iterations. Additionally, we set the maximum number of iterations to 500.
\subsection{Common examples}
In this subsection, we compare PGMO, APGMO, SPGMO and ASPGMO with line search across several common examples. Each objective function in the tested problems consists of two components, where $g_i=\frac{1}{n}\|x\|_{1}$ for $i\in[m]$, and the details of $f$ are provided in Table \ref{tab2}. The recorded averages from 200 runs include the number of iterations, the number of function evaluations, and the CPU time.
\begin{small}
	\begin{table}[h]
		\centering
		\caption{Description of all test problems used in numerical experiments.}
		\label{tab2}
		\resizebox{.7\columnwidth}{!}{
			\begin{tabular}{lllllllllll}
				\hline
				Problem    &  & $n$     &           & $m$     &  & $x_{L}$               &  & $x_{U}$             &  & Reference \\ \hline
				DD1        &  & 5     & \textbf{} & 2     &  & (-20,...,-20)   &  & (20,...,20)   &  & \cite{DD1998}         \\
				Deb        &  & 2     & \textbf{} & 2     &  & (0.1,0.1)       &  & (1,1)         &  & \cite{D1999}          \\
				Far1       &  & 2     & \textbf{} & 2     &  & (-1,-1)         &  & (1,1)         &  & \cite{BK1}         \\
				FDS        &  & 5     & \textbf{} & 3     &  & (-2,...,-2)     &  & (2,...,2)     &  & \cite{FD2009}         \\
				FF1        &  & 2     & \textbf{} & 2     &  & (-1,-1)         &  & (1,1)         &  & \cite{BK1}         \\
				Hil1       &  & 2     & \textbf{} & 2     &  & (0,0)           &  & (1,1)         &  & \cite{Hil1}         \\
				Imbalance1 &  & 2     & \textbf{} & 2     &  & (-2,-2)         &  & (2,2)         &  & \cite{CTY2023}         \\
				Imbalance2 &  & 2     & \textbf{} & 2     &  & (-2,-2)         &  & (2,2)         &  & \cite{CTY2023}         \\
				VU1        &  & 2     & \textbf{} & 2     &  & (-3,-3)         &  & (3,3)         &  & \cite{BK1}        \\
				WIT1       &  & 2     & \textbf{} & 2     &  & (-2,-2)         &  & (2,2)         &  & \cite{W2012}       \\
				WIT2       &  & 2     & \textbf{} & 2     &  & (-2,-2)         &  & (2,2)         &  & \cite{W2012}        \\
				WIT3       &  & 2     & \textbf{} & 2     &  & (-2,-2)         &  & (2,2)         &  & \cite{W2012}        \\
				\hline
			\end{tabular}
		}
	\end{table}
\end{small}

%
%
%
\begin{table}[h]
	\centering
	\caption{Number of average iterations (iter), number of average function evaluations (feval) and average CPU time (time ($ms$)) of test algorithms on different test problems {\bf with line search}.}
	\label{tab3}
	\resizebox{.99\columnwidth}{!}{
		\begin{tabular}{lrrrrrrrrrrrrrrr}
			\hline
			Problem &
			\multicolumn{3}{l}{PGMO} &
			&
			\multicolumn{3}{l}{APGMO} &
			\multicolumn{1}{l}{} &
			\multicolumn{3}{l}{SPGMO} &
			\multicolumn{1}{l}{} &
			\multicolumn{3}{l}{ASPGMO} \\ \cline{2-4} \cline{6-8} \cline{10-12} \cline{14-16}
			&
			iter &
			feval &
			time &
			\textbf{} &
			iter &
			feval &
			time &
			&
			iter &
			feval &
			time &
			&
			iter &
			feval &
			time \\ \hline
			DD1        & 41.21 & 70.65 & 135.31 &  & 67.75 & 132.46 & 400.16 &  & \textbf{4.52} & \textbf{4.90} & \textbf{31.02} &  & 6.43 & 7.27 & 65.31 \\
			Deb        & 26.52 & 253.17 & 83.67 &  & 14.39 & 30.44 & 58.75 &  & 5.97 & 9.86 & \textbf{28.83} &  & \textbf{5.15} & \textbf{8.04} & 31.80 \\
			Far1       & \textbf{6.06} & 21.09 & 16.88 &  & 8.06 & 17.75 & 29.38 &  & 6.76 & \textbf{7.86} & \textbf{9.92} &  & {7.97} & {14.41} & {19.53} \\
			FDS        & 175.75 & 778.54 & 1350.16 &  & 55.15 & 107.88 & 737.89 &  & \textbf{3.44} & \textbf{3.82} & \textbf{2.16} &  & 7.36 & 12.00 & 54.53 \\
			FF1        & 3.41 & {3.55} & 2.42 &  & 4.13 & 5.44 & 3.05 &  & \textbf{2.10} & \textbf{2.26} & \textbf{2.27} &  & {2.93} & {4.00} & 5.55 \\
			Hil1       & 8.71 & 18.62 & 27.50 &  & \textbf{7.02} & 15.81 & 30.07 &  & 7.49 & \textbf{8.26} & \textbf{13.91} &  & {7.88} & {12.24} & {25.00} \\
			Imbalance1 & {2.92} & 5.63 & 42.81 &  & 7.79 & 13.80 & 134.53 &  & \textbf{2.44} & \textbf{3.11} & \textbf{34.69} &  & {3.82} & {5.35} & 39.14 \\
			Imbalance2 & 83.56 & 589.13 & 1180.39 &  & 21.09 & 42.18 & 149.84 &  & \textbf{1.00} & \textbf{1.00} & \textbf{1.79} &  & \textbf{1.00} & \textbf{1.00} & 2.65 \\
			VU1        & 8.52 & 8.59 & \textbf{1.72} &  & 6.20 & 10.47 & 9.69 &  & \textbf{2.08} & \textbf{2.15} & {3.36} &  & {3.61} & {5.31} & {7.42} \\
			WIT1       & {27.63} & 145.06 & 118.98 &  & 20.05 & 41.14 & 51.64 &  & \textbf{2.94} & \textbf{3.24} & \textbf{7.11} &  & {4.83} & {6.71} & 15.73 \\
			WIT2       & {48.07} & 286.20 & 193.67 &  & 19.22 & 39.65 & 83.36 &  & \textbf{3.14} & \textbf{3.34} & \textbf{9.84} &  & {4.91} & {6.49} & 19.69 \\
			WIT3       & {18.58} & 79.78 & 74.30 &  & 35.51 & 71.72 & 124.68 &  & \textbf{3.92} & \textbf{4.13} & \textbf{14.22} &  & {5.33} & {7.64} & 17.89 \\
			\hline
		\end{tabular}
	}
\end{table}
\par Table \ref{tab3} presents the average number of iterations (iter), average number of function evaluations (feval) and average CPU time (time ($ms$)) for each tested algorithm across various problems with line search. In both the non-accelerated and accelerated cases, the numerical results indicate that SPGMO consistently outperforms PGMO. Additionally, PGMO demonstrates poor performance on problems such as DD1, Deb, FDS, imbalance2, and WIT1-3, which exhibit imbalanced objective functions. Given the performance of SPGMO on these problems, it is well-suited for addressing such challenges. Table \ref{tab3} also shows that APGMO surpasses PGMO, validating the effectiveness of Nesterov's acceleration technique in MOPs. However, we note that Nesterov's acceleration technique slightly hampers the performance of SPGMO. We infer that the primary reasons for this may be the Barzilai-Borwein method utilized in SPGMO and the simplicity of the problems tested.
\subsection{Quadratic examples}
In the following, we compare different test problem with known smooth parameters. Consider a series of quadratic problems defined as follows:
$$f_{i}(x)=\frac{1}{2}\left\langle x,A_{i}x\right\rangle + \left\langle b_{i},x\right\rangle,\ g_{i}(x)=\frac{1}{n}\|x\|_{1},~i=1,2,$$
where $A_i\in\mathbb{R}^{n\times n}$ is a positive definite matrix. Here we use, $A_{i}=H_{i}D_{i}H_{i}^{T}$ where $H_{i}$ is a
orthogonal matrix and $D_{i}$ is a
diagonal matrix with positive components. More details of $f$ are provided in Table \ref{tab2}, where $\kappa$ and $\zeta$ denote the condition number and objective imbalances parameter, respectively. QPa is a well-conditioned problem without objective imbalances, while QPb-f exhibit objective imbalances with a parameter value of 100. Additionally, QPd and QPf display significant ill-conditioning, which poses challenges for first-order methods. For each problem, we conducted 200 computations using the same initial points across the different tested algorithms. The initial points were randomly selected within the specified lower and upper bounds. The recorded averages from the 200 runs include the number of iterations and the CPU time.
\begin{table}[h]
	\centering
	\caption{Description of quadratic problems}
	\label{tab1}
	
	\begin{tabular}{clrlclrlrl}
		\hline
		\multicolumn{1}{l}{Problem} &  & n   &           & $(\kappa,~\zeta)$ &           & \multicolumn{1}{c}{$x_{L}$} &  & \multicolumn{1}{c}{$x_{U}$} &  \\ \hline
		QPa                         &  & 10  &           & $(10,1)$ &           & 10{[}-1,...,-1{]}      &  & 10{[}1,...,1{]}        &  \\
		QPb                         &  & 10  &           & $(10,10^{2})$ &           & 10{[}-1,...,-1{]}      &  & 10{[}1,...,1{]}        &  \\
		QPc                         &  & 10 &           & $(10^{2},10^{2})$ &           & 10{[}-1,...,-1{]}     &  & 10{[}1,...,1{]}       &  \\
		QPd                         &  & 10 &           & $(10^{4},10^{2})$ &           & 10{[}-1,...,-1{]}     &  & 10{[}1,...,1{]}       &  \\
		QPe                         &  & 100 &           & $(10^{2},10^{2})$ &           & 100{[}-1,...,-1{]}     &  & 100{[}1,...,1{]}       &  \\
		QPf                         &  & 100 & \textbf{} & $(10^{3},10^{2})$ & \textbf{} & 100{[}-1,...,-1{]}     &  & 100{[}1,...,1{]}       &  \\
		\hline
	\end{tabular}
\end{table}


\begin{table}[h]
	\centering
	\caption{Number of average iterations (iter) and average CPU time (time ($ms$)) of tested algorithms implemented on different test problems {\bf with known smooth parameters}.}
	\label{tab4}
	\resizebox{.99\columnwidth}{!}{
		\begin{tabular}{llrrrrrrrrrrrrrrrrr}
			\hline
			{Problem} &
			&
			\multicolumn{2}{l}{PGMO} &
			\multicolumn{1}{l}{} &
			\multicolumn{2}{l}{APGMO} &
			\multicolumn{1}{l}{} &
			\multicolumn{2}{l}{APGMO-sc} &
			\multicolumn{1}{l}{} &
			\multicolumn{2}{l}{SPGMO} &
			\multicolumn{1}{l}{} &
			\multicolumn{2}{l}{ASPGMO} &
			\multicolumn{1}{l}{} &
			\multicolumn{2}{l}{ASPGMO-sc} \\ \cline{3-4} \cline{6-7} \cline{9-10} \cline{12-13} \cline{15-16} \cline{18-19} 
			&  & iter   & time   &  & iter   & time    &  & iter   & time    &  & iter   & time   &  & iter   & time   &  & iter   & time   \\ \hline
			QPa &  & 43.07  & 94.92  &  & 33.42  & 99.14   &  & \bf{20.56}  & 60.47   &  & 43.07  & 100.31 &  & 33.42  & 105.16 &  & \bf{20.56}  & \bf{59.53}  \\
			QPb &  & -- & -- &  & -- & --  &  & 196.03 & 477.42  &  & 48.44  & 89.21  &  & 34.30  & 82.03  &  & \bf{21.21}  & \bf{62.81}  \\
			QPc &  & -- & -- &  & -- & --  &  & -- & --  &  & 367.21 & 210.78 &  & 149.33 & 144.53 &  & \bf{68.97}  & \bf{74.22}  \\
			QPd &  & -- & -- &  & -- & --  &  & -- & --  &  & -- & --  &  & -- & --  &  & \bf{422.72} & \bf{25.70}  \\
			QPe &  & -- & --   &  & -- & -- &  & -- & -- &  & 326.31 & 368.20 &  & 186.47 & 320.54 &  & \bf{81.66}  & \bf{219.61} \\
			QPf &  &    --    &   --     &  &     --   &     --    &  &     --   &      --   &  & -- & -- &  & -- & -- &  & \bf{262.87} & \bf{351.25} \\ \hline
		\end{tabular}
	}
\end{table}

Table \ref{tab4} provides the average number of iterations (iter) and average CPU time (time ($ms$)) for each tested algorithm across the various problems with known smooth parameters, where ``--'' indicates that the tested algorithm failed to meet the stopping criteria within $500$ iterations. As noted in Table \ref{tab4}, we conclude that PGMO exhibits slow convergence on imbalanced problems in both the non-accelerated and accelerated scenarios. The subpar performance of APGMO on imbalanced problems further confirms that Nesterov's acceleration technique is inadequate for addressing objective imbalances. In comparing the performances of SPGMO and ASPGMO, we conclude that the performance of SPGMO with known smooth parameters is enhanced by the incorporation of Nesterov's acceleration technique. Remarkably, even for the highly ill-conditioned problems QPd and QPf, the ASPGMO with strongly convex momentum (ASPGMO-sc) shows commendable performance.

\begin{figure}[H]
	\centering
	\subfigure[$k_{\max}=50$]
	{
		\begin{minipage}[H]{.45\linewidth}
			\centering
			\includegraphics[scale=0.37]{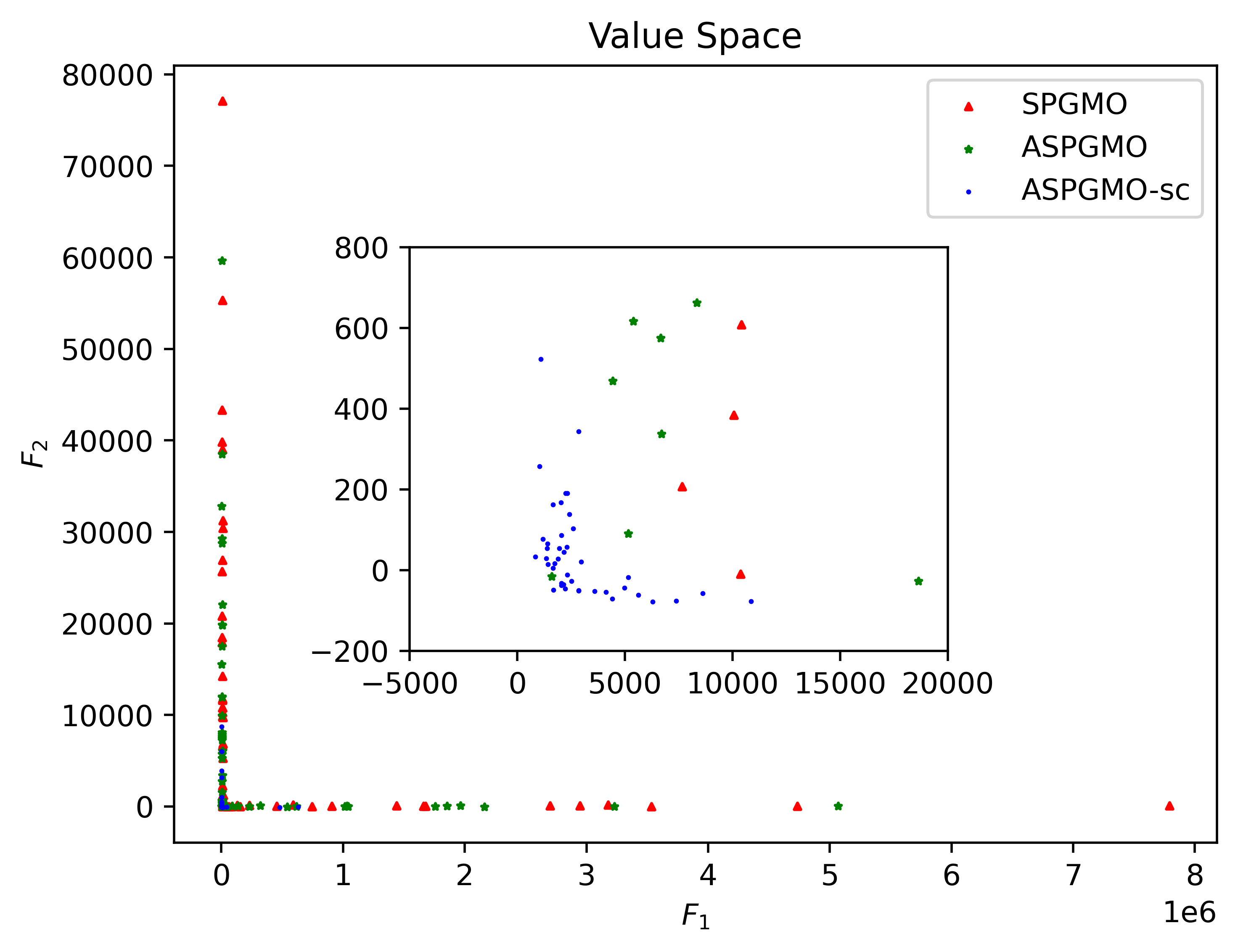} 
		\end{minipage}
	}
	\subfigure[$k_{\max}=500$]
	{
		\begin{minipage}[H]{.45\linewidth}
			\centering
			\includegraphics[scale=0.37]{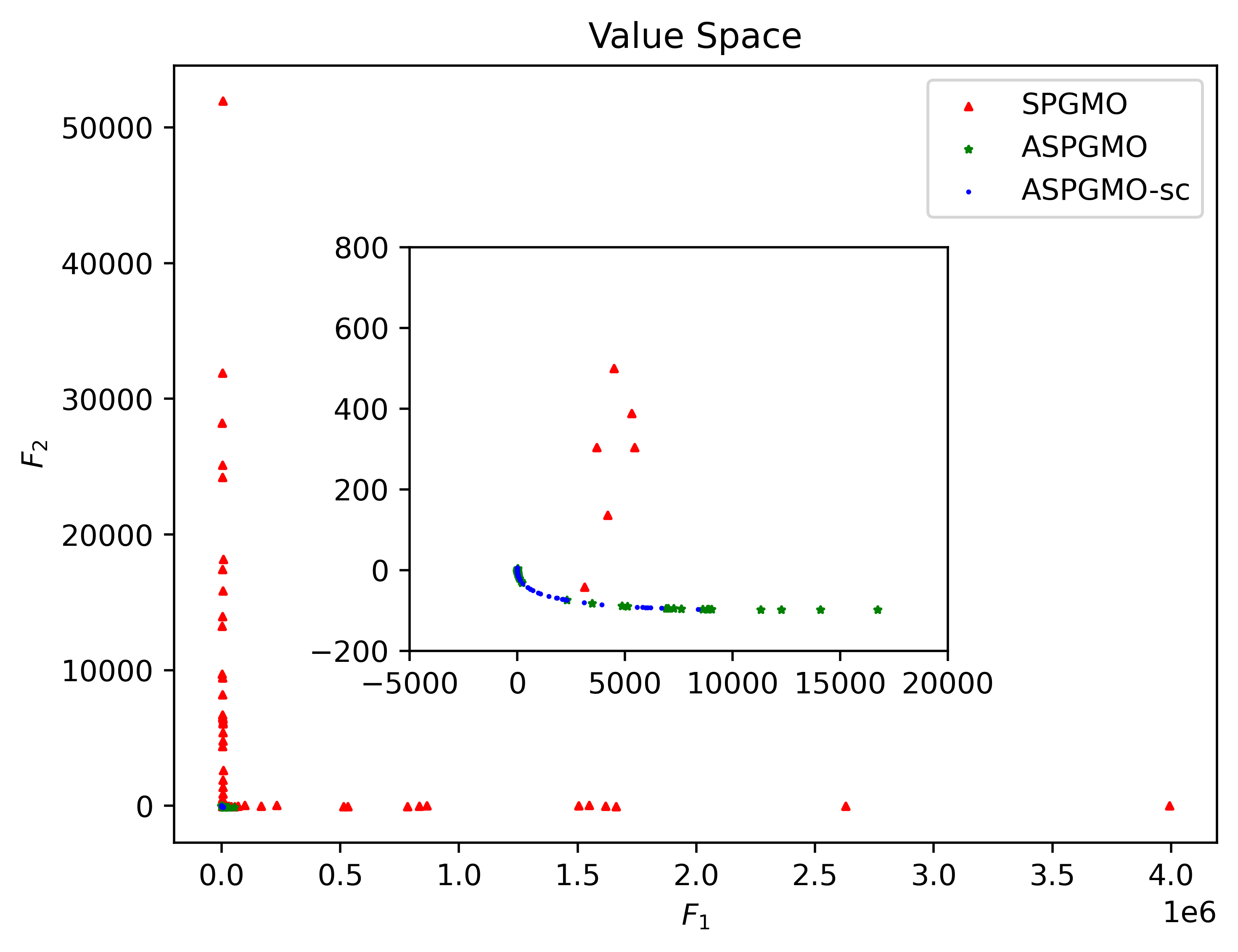} 
		\end{minipage}
	}
	
	\caption{Value space for different maximum numbers of iterations
		$k_{\max}=50,500$ for QPd.}
	\label{f0}
\end{figure}

\begin{figure}[H]
	\centering
	\subfigure[Function values of $F_{1}$]
	{
		\begin{minipage}[H]{.45\linewidth}
			\centering
			\includegraphics[scale=0.37]{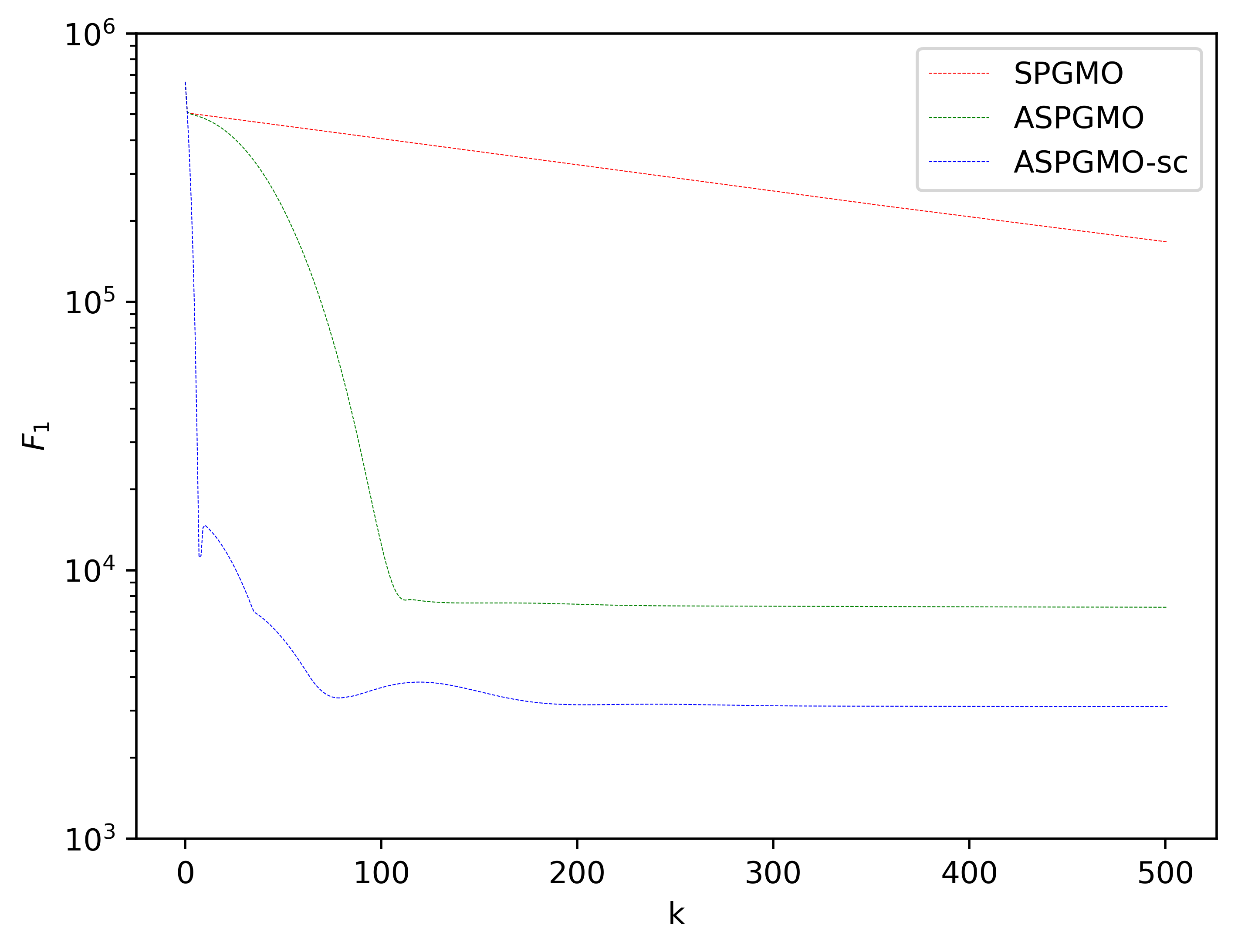} 
		\end{minipage}
	}
	\subfigure[Function values of $F_{2}$]
	{
		\begin{minipage}[H]{.45\linewidth}
			\centering
			\includegraphics[scale=0.37]{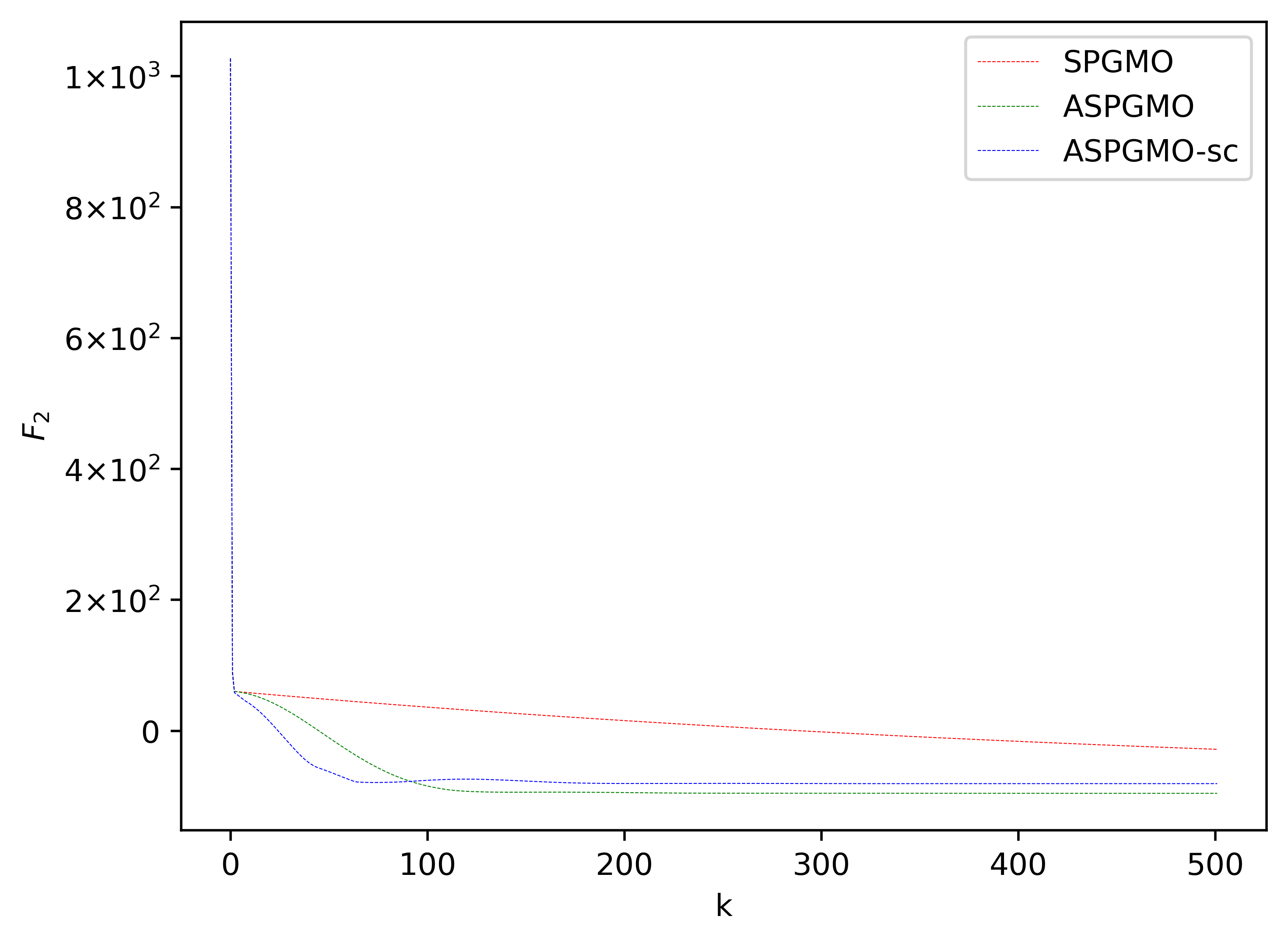} 
		\end{minipage}
	}
	\caption{Function values of iterations for QPd.}
	\label{f1}
\end{figure}

\begin{figure}[H]
	\centering
	\subfigure[Function values of $F_{1}$]
	{
		\begin{minipage}[H]{.45\linewidth}
			\centering
			\includegraphics[scale=0.37]{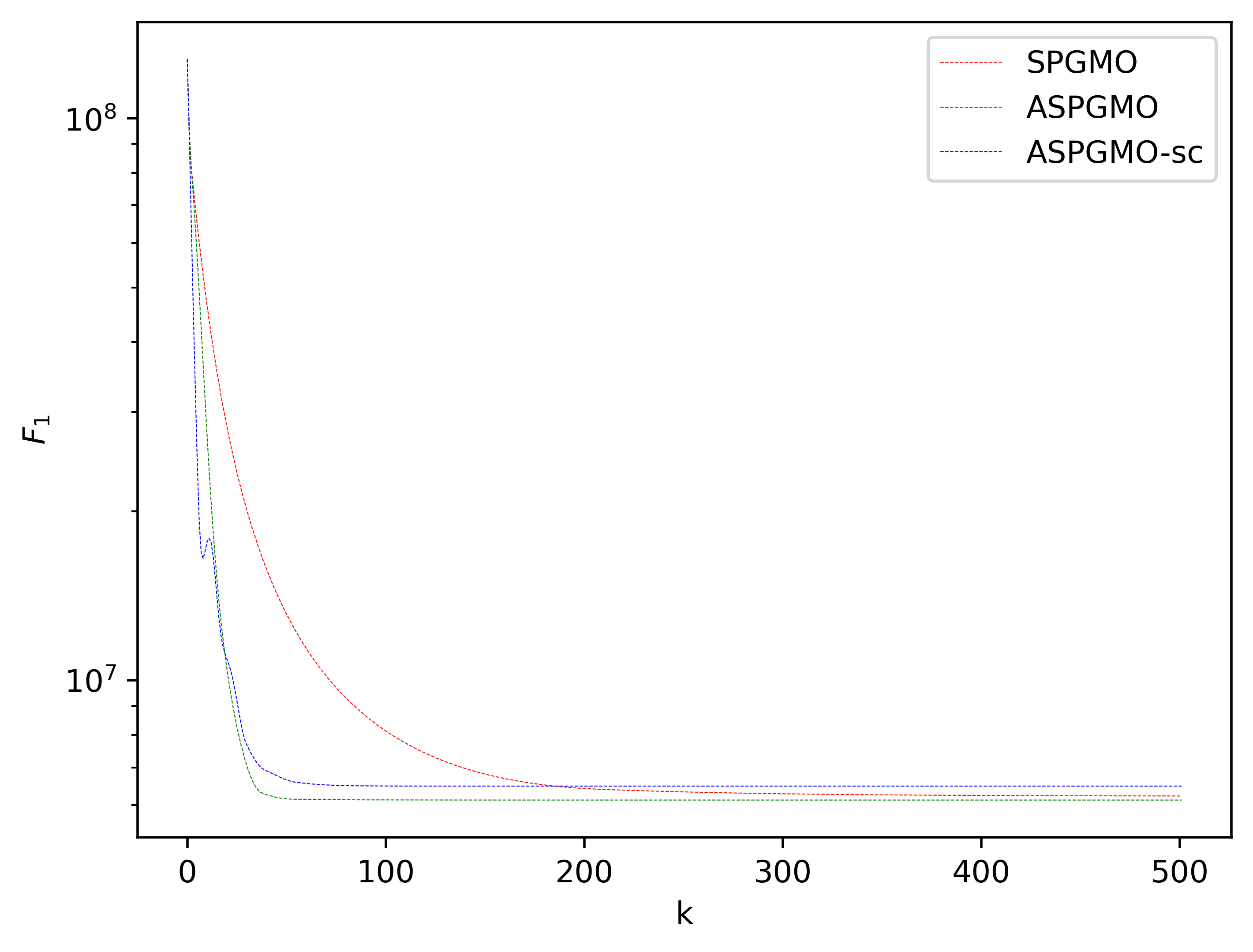} 
		\end{minipage}
	}
	\subfigure[Function values of $F_{2}$]
	{
		\begin{minipage}[H]{.45\linewidth}
			\centering
			\includegraphics[scale=0.37]{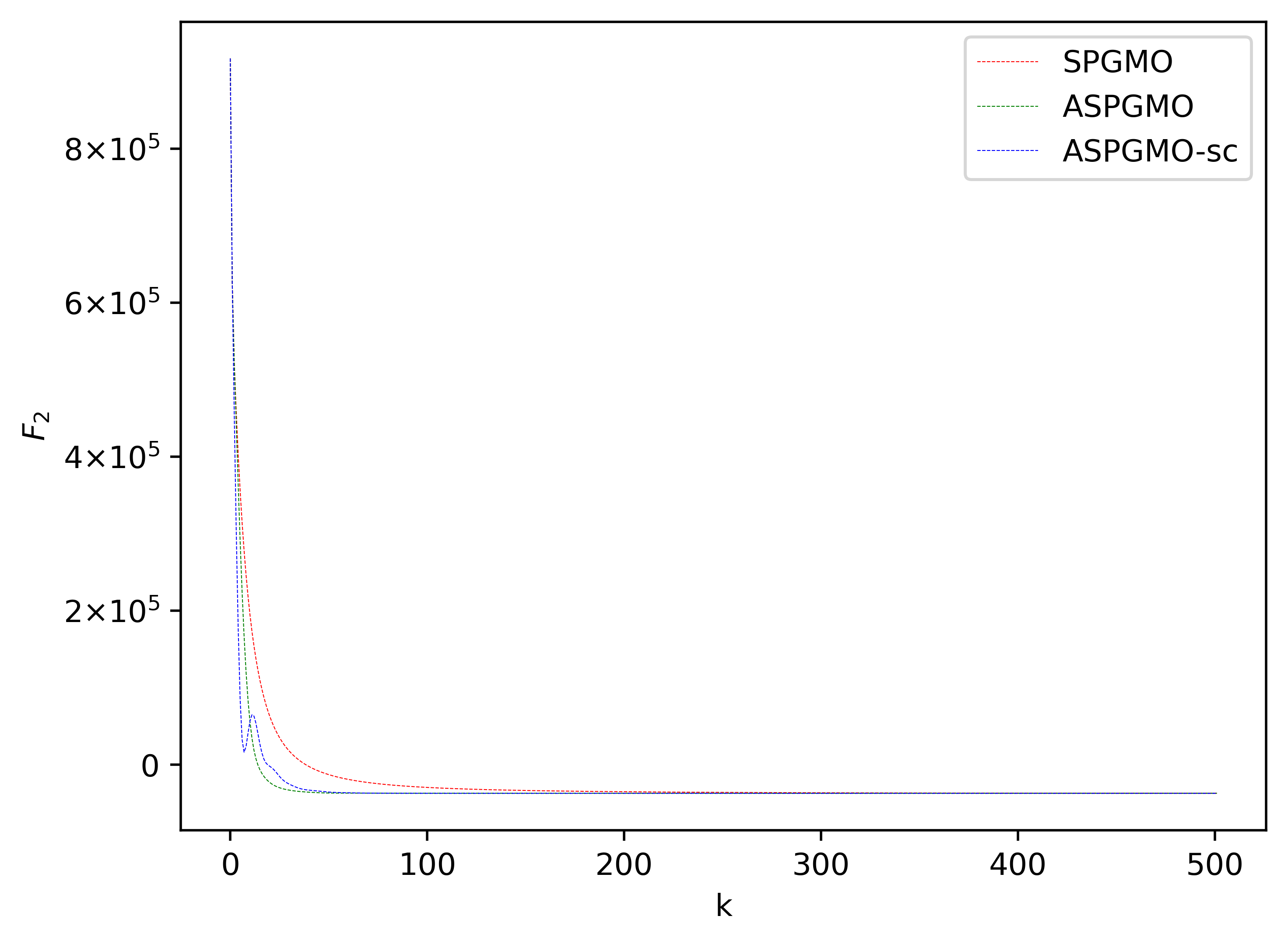} 
		\end{minipage}
	}
	\caption{Function values of iterations for QPf.}
	\label{f2}
\end{figure}

In Figure \ref{f0}, we utilize the same 50 random starting points for SPGMO, ASPGMO, and ASPGMO-sc. For these experiments, we apply different maximal numbers of iterations $k_{\max}$. As shown in Figure \ref{f0} (a), the points generated by ASPGMO-sc progress significantly faster than those produced by SPGMO and ASPGMO in the initial stages. In Figure \ref{f0} (b), we observe that the points generated by SPGMO are consistently outperformed by both ASPGMO and ASPGMO-sc. Although SPGMO and ASPGMO do not meet the stopping criteria within 500 iterations on the QPe and QPf problems, Figures \ref{f1} and \ref{f2} indicate that the function values produced by ASPGMO decrease more rapidly at the outset compared to SPGMO. Moreover, from Figures \ref{f1} and \ref{f2}, we note that the function values of ASPGMO and ASPGMO-sc exhibit a slower decrease after 100 iterations and converge to similar values. However, as illustrated in Figure \ref{f3}, both ASPGMO and ASPGMO-sc converge more rapidly than the non-accelerated SPGMO, despite some oscillations. Additionally, ASPGMO-sc demonstrates superior performance compared to ASPGMO, successfully reaching the stopping criteria within 500 iterations for the QPd and QPf problems.

\begin{figure}[t]
	\centering
	\subfigure[$\nm{x^{k+1}-y^{k}}$ for problem QPd]
	{
		\begin{minipage}[H]{.45\linewidth}
			\centering
			\includegraphics[scale=0.37]{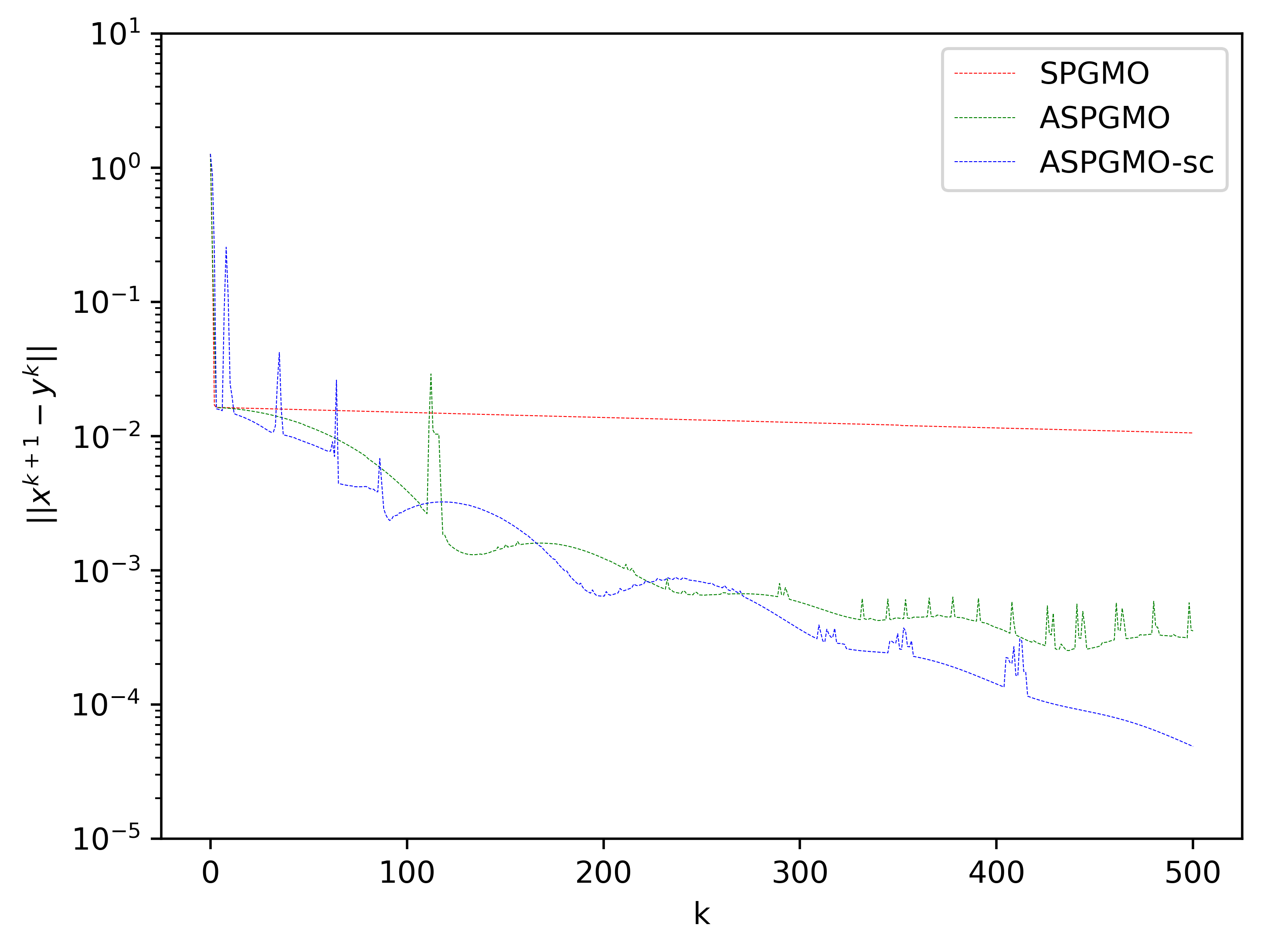} 
		\end{minipage}
	}
	\subfigure[$\nm{x^{k+1}-y^{k}}$ for problem QPf]
	{
		\begin{minipage}[H]{.45\linewidth}
			\centering
			\includegraphics[scale=0.37]{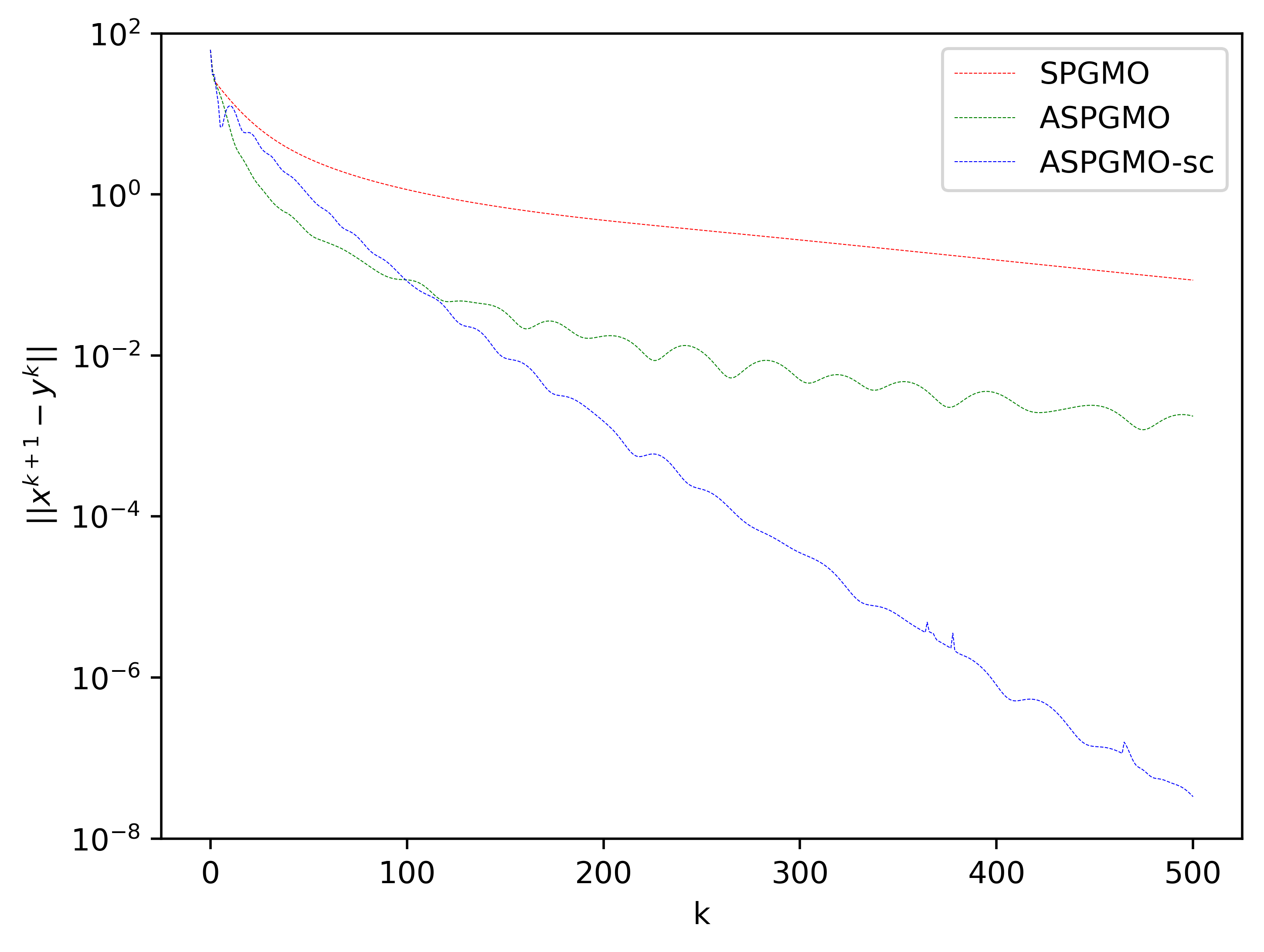} 
		\end{minipage}
	}
	\caption{$\nm{x^{k+1}-y^{k}}$ for problems QPd and QPf.}
	\label{f3}
\end{figure}

\section{Conclusions}
We develop a scaled method to mitigate objective imbalances that lead to the slow convergence of existing multiobjective first-order methods. It is proven that the SPGMO converges linearly at a rate of $\sqrt{1-\min_{i\in[m]}\left\{{\mu_{i}}/{L_{i}}\right\}}$, whereas the linear convergence rate of PGMO is $\sqrt{1-{\mu_{\min}}/{L_{\max}}}$. We also establish improved linear convergence in both linear and accelerated cases. To the best of our knowledge, the improved linear convergence is the first theoretical result that bridges the theoretical gap between first-order methods for SOPs and MOPs. Notably, scaling serves as a flexible manipulation that opens new avenues for exploring efficient descent methods in multiobjective optimization.
\par For future work, several points warrant consideration:
\begin{itemize}
	\item Objective imbalances are intrinsic to MOPs that reflect differences among objectives. This paper quantifies objective imbalances by (\ref{oimb}) for first-order methods in strongly convex cases. It would be interesting to study objective imbalances in non-convex scenarios as well as in high-order methods.
	\item The line search method described in Remark \ref{r5.6} (iii) suggests selecting appropriate scaling parameters to capture the problem's local geometry. However, line search methods often impose significant computational burdens. By extending the fully adaptive method introduced in \cite{Ada20} to determine scaling parameters, developing an adaptive SPGMO could be an interesting avenue for exploration. 
	\item Numerical results indicate that the ASPGMO performs well for ill-conditioned MOPs when the strongly convex parameters are known. However, these parameters are frequently unknown and challenging to estimate. It would be worthwhile to develop efficient ASPGMO algorithms that do not require knowledge of the strongly convex parameters. 
\end{itemize}

\bibliographystyle{abbrv}
\bibliography{references}

\begin{acknowledgements}
This work was funded by the Major Program of the National Natural Science Foundation of China [grant numbers 11991020, 11991024]; the Key Program of the National Natural Science Foundation of China [grant number 12431010], the General Program of the National Natural Science Foundation of China [grant number 12171060]; NSFC-RGC (Hong Kong) Joint Research Program [grant number 12261160365].
\end{acknowledgements}

\end{document}